\documentclass[a4paper,11pt]{smfart}

\usepackage[applemac]{inputenc}
\RequirePackage[T1]{fontenc}
\RequirePackage{amsfonts,latexsym,amssymb}
\RequirePackage[frenchb]{babel}
\addto\extrasfrenchb{\bbl@nonfrenchitemize
\bbl@nonfrenchspacing}

\RequirePackage{mathrsfs}
\let\mathcal\mathscr

\RequirePackage{smfenum}

\newtheoremstyle{theorem}{11pt}{11pt}{\slshape}{}{\bfseries}{.}{.5em}{}
\newtheoremstyle{note}{11pt}{11pt}{}{}{\bfseries}{.}{.5em}{}

\theoremstyle{plain}
  \newtheorem{theorem}{Th\'{e}or\`{e}me}[section]
  \newtheorem{proposition}[theorem]{Proposition}
  \newtheorem{lemma}[theorem]{Lemme}
  \newtheorem{corollary}[theorem]{Corollaire}

\theoremstyle{definition}
  \newtheorem{definition}[theorem]{D\'{e}finition}

\theoremstyle{remark}
  
  \newtheorem{remark}[theorem]{Remarque}

\usepackage[matrix,arrow]{xy}
\usepackage{url}

\let\cal\mathcal

\def\Q{{\bf Q}} \def\Z{{\bf Z}}

\def\O{{\cal O}}

\def\dual{{\boldsymbol *}}

\def\zp{{\Z_p}}
\def\zpet{{\Z_p^\dual}}
\def\qp{{\Q_p}}
\def\qpet{{\Q_p^\dual}}
\def\p1{{\bf P}^1}

\def\epsilon{\varepsilon}

\begin{document}
\title[Extensions et vecteurs algébriques]
{Extensions de représentations de de Rham et vecteurs localement algébriques}
\author{Gabriel Dospinescu}
\address{UMPA, \'Ecole Normale Sup\'erieure de Lyon, 46 all\'ee d'Italie, 69007 Lyon, France}
\email{gabriel.dospinescu@ens-lyon.fr}
\begin{abstract}
  Soit $\Pi$ une complétion unitaire irréductible d'une représentation localement algébrique de ${\rm GL}_2(\qp)$.
  On décrit les déformations infinitésimales $\Pi_1$ de $\Pi$ qui sont elles-mêmes complétions d'une représentation localement algébrique.
  Cela répond à une question de Paskunas et a des applications directes
 à la conjecture de Breuil-Mézard. 
 
 
\end{abstract}
\begin{altabstract}
Let $\Pi$ be an irreducible unitary completion of a locally algebraic ${\rm GL}_2(\qp)$-representation. 
We describe those first-order deformations of $\Pi$ which are themselves completions of a locally algebraic representation.
This answers a question of Paskunas and has direct applications to the Breuil-Mézard conjecture. 


\end{altabstract}
\setcounter{tocdepth}{3}

\maketitle

\stepcounter{tocdepth}
{\Small
\tableofcontents
}

\section{Introduction}

  \subsection{Le résultat principal} \quad Soit $p$ un nombre premier, $L$ une extension finie de $\qp$, $O_L$ l'anneau de ses entiers et soit $\delta: \qpet\to O_L^*$ un caractère unitaire. On note 
   ${\rm Rep}_L(\delta)$ la catégorie des représentations de $G={\rm GL}_2(\qp)$ sur des $L$-espaces de Banach $\Pi$,  
   telles que 
   
   $\bullet$ $\Pi$ a pour caractère central $\delta$.
   
   $\bullet$ $\Pi$ est unitaire, i.e. il existe une valuation $v_{\Pi}$ qui définit la topologie de $\Pi$ et qui est $G$-invariante.
   
   $\bullet$ Si $\Pi_0$ est la boule unité de $\Pi$ pour la valuation $v_{\Pi}$, alors $\Pi_0/p\Pi_0$ est un $O_L[G]$-module 
   de longueur finie.  
   
     Un objet $\Pi$ de ${\rm Rep}_L(\delta)$ est dit {\it supersingulier} si $\Pi$ est absolument irréductible\footnote{Cela signifie que $\Pi\otimes_{L} L'$ est topologiquement
     irréductible pour toute extension finie $L'$ de $L$.} et si $\Pi$ n'est pas isomorphe à un sous-quotient de
     l'induite parabolique d'un caractère unitaire du tore diagonal de $G$. Si $\Pi\in {\rm Rep}_L(\delta)$, on note 
     $\Pi^{\rm alg}$ l'espace des vecteurs localement algébriques de $\Pi$. Ce sont les vecteurs $v$ dont l'application orbite
     $g\to g\cdot v$ est localement polynômiale sur $G$. $\Pi^{\rm alg}$ est la plupart du temps nul\footnote{Si $\Pi^{\rm alg}\ne 0$, alors 
     $\delta$ est localement algébrique, mais cela est loin d'être suffisant.} et on a 
     $\Pi^{\rm alg}\ne 0$ si et seulement si $\Pi$ contient une représentation de la forme 
     $W\otimes \pi$, avec $W$ une représentation algébrique de $G$ et 
     $\pi$ une représentation lisse, admissible (bien entendu, avec $W\ne 0$ et $\pi\ne 0$).
     Les représentations $\Pi$ telles que $\Pi^{\rm alg}\ne 0$ jouent un rôle très important à travers les connexions 
     qu'elles ont avec la théorie des représentations galoisiennes: leur étude est un ingrédient clé dans la preuve des conjectures
     de Fontaine-Mazur \cite{KiFM, Emcomp} et de Breuil-Mézard \cite{KiFM, BM, PaskunasBM}. Leur étude a été initiée par 
     Breuil \cite{Br}, ce qui a abouti \cite{Cbigone, Pa} à la correspondance de Langlands locale $p$-adique pour $G$.

     \begin{remark}\label{extalg} Si $\Pi$ est irréductible et si $\Pi^{\rm alg}\ne 0$, alors 
     $\Pi^{\rm alg}$ est dense dans $\Pi$ (puisque $\Pi^{\rm alg}$ est stable par $G$), donc $\Pi$ est une complétion unitaire de $\Pi^{\rm alg}$.
     Si $\Pi$ est réductible et si $\Pi^{\rm alg}\ne 0$, il n'y a aucune raison pour que $\Pi^{\rm alg}$ soit dense dans 
     $\Pi$. L'exemple suivant est assez éclairant: soient $\delta_1,\delta_2$ des caractères unitaires tels que 
     $\delta_1\delta_2=\chi\delta$, où $\chi:{\rm Gal}(\overline{\qp}/\qp)\to \zpet$ est le caractère cyclotomique. 
     On suppose que $\delta_1\delta_2^{-1}\ne 1,\chi^{\pm 1}$ et on considère les induites paraboliques continues\footnote{Si 
     $\delta_1\otimes \delta_2$ est un caractère unitaire du tore diagonal de ${\rm GL}_2(\qp)$, on note 
     ${\rm Ind}_B^{G}(\delta_1\otimes \delta_2)^{\rm cont}$ l'espace des fonctions continues $f: {\rm GL}_2(\qp)\to L$
     telles que $f(\left(\begin{smallmatrix} a & b \\0 & d\end{smallmatrix}\right)g)=\delta_1(a)\delta_2(d)f(g)$ pour 
     $a,d\in \qpet$, $b\in \qp$ et $g\in {\rm GL}_2(\qp)$. C'est naturellement un objet de ${\rm Rep}_L(\delta_1\delta_2)$.}
           $\Pi_1= {\rm Ind}_{B}^{G} (\delta_1\otimes \delta_2\chi^{-1})^{\rm cont}$ et
     $\Pi_2={\rm Ind}_{B}^{G} (\delta_2\otimes \delta_1\chi^{-1})^{\rm cont}$. Alors 
     $\Pi_1,\Pi_2\in {\rm Rep}_L(\delta)$ sont des objets absolument irréductibles non isomorphes (mais pas supersinguliers non plus) et il existe\footnote{Ce genre d'extension n'existerait pas si on travaillait avec des coefficients $l$-adiques, où $l\ne p$ est un nombre premier.} une extension
     nontriviale $0\to \Pi_1\to \Pi \to \Pi_2\to 0$ dans ${\rm Rep}_L(\delta)$. On montre \cite[lemma 11.5]{Pa} que si 
     $\Pi_1^{\rm alg}\ne 0$, alors $\Pi_2^{\rm alg}=0$, donc $\Pi^{\rm alg}=\Pi_1^{\rm alg}$ n'est pas dense dans $\Pi_1$.

    
     \end{remark}

       Colmez a défini \cite[ch. IV]{Cbigone} un foncteur exact contravariant $\Pi\to D(\Pi)$ de ${\rm Rep}_L(\delta)$ 
       dans la catégorie $\Phi\Gamma^{\rm et}(\mathcal{E})$ des $(\varphi,\Gamma)$-modules étales sur un corps de séries
       de Laurent $\mathcal{E}$. D'après Fontaine \cite{FoGrot}, la catégorie $\Phi\Gamma^{\rm et}(\mathcal{E})$ est équivalente à 
   la catégorie ${\rm Rep}_L(G_{\qp})$ des $L$-représentations continues de dimension finie de ${\rm Gal}(\overline{\qp}/\qp)$, d'où un foncteur (contravariant et exact)
   $\Pi\to V(\Pi)$ de ${\rm Rep}_L(\delta)$ dans ${\rm Rep}_L(G_{\qp})$. On suppose dans la suite de l'introduction que $p\geq 5$. Soit 
   $\Pi\in {\rm Rep}_L(\delta)$ un objet supersingulier. On dispose alors de deux résultats suivants, qui ont fait couler beaucoup d'encre:
   
   $\bullet$ $V(\Pi)$ est absolument irréductible, de dimension $2$, et de déterminant\footnote{Nous ne suivons pas les normalisations de Paskunas, mais plutôt 
     celles de Colmez \cite{Cbigone}. On passe d'une normalisation à l'autre par la dualité de Cartier 
     $V\to V^*(1)$.} $\zeta:=\chi\cdot \delta^{-1}$, où
   $\chi:{\rm Gal}(\overline{\qp}/\qp)\to \zpet$ est le caractère cyclotomique. 
   
   $\bullet$ $\Pi^{\rm alg}\ne 0$ si et seulement si $V(\Pi)$ est de de Rham, à poids de Hodge-Tate distincts.
   
     Le premier résultat est dû à Paskunas \cite[th. 11.4]{Pa}, et le second suit du premier et d'un théorème de Colmez \cite[th. VI.6.13, VI.6.18]{Cbigone}. 
      Le but de cet article est de montrer le résultat suivant, qui répond à une question de Paskunas et peut être vu comme une version du théorème
      de Colmez pour les déformations infinitésimales de $\Pi$.
      
      \begin{theorem}\label{main}  Soit $p\geq 5$ et soit 
      $\Pi\in {\rm Rep}_L(\delta)$ un objet supersingulier, 
      tel que $\Pi^{\rm alg}\ne 0$. On considère une suite exacte $0\to \Pi\to \Pi_1\to \Pi\to 0$ dans 
      ${\rm Rep}_L(\delta)$. Alors les assertions suivantes sont équivalentes:
      
      a) $\Pi_1^{\rm alg}$ est dense dans $\Pi_1$.
      
      b) $\Pi_1^{\rm alg}\ne \Pi^{\rm alg}$.
      
      c) $V(\Pi_1)$ est de de Rham.
            
      \end{theorem}

      \subsection{Lien avec la conjecture de Breuil-Mézard et généralisation d'un résultat de Colmez}
      
        Nous allons maintenant expliquer pourquoi le théorème \ref{main} finit la preuve de \cite[th. 4.18]{PaskunasBM}, qui est un des ingrédients techniques 
        majeurs de la preuve locale de Paskunas de la conjecture 
       de Breuil-Mézard \cite{BMorig}. 
    Soit $\Pi\in {\rm Rep}_L(\delta)$ un objet supersingulier (on ne suppose pas pour l'instant que $\Pi^{\rm alg}\ne 0$) et 
        notons\footnote{Cette catégorie est note ${\rm Ban}_{G,\zeta}^{\rm adm.fl}(L)_{\Pi}$ dans \cite{Pa}.} ${\rm Rep}_L(\delta)_{\Pi}$ la 
          sous-catégorie pleine de ${\rm Rep}_L(\delta)$ formée des objets dont tous les facteurs de Jordan-Hölder sont isomorphes à $\Pi$. Soit $R^{\zeta}$
          l'anneau de déformation universel de $V(\Pi)$ avec déterminant $\zeta$ (on a un isomorphisme 
          $R^{\zeta}\simeq O_L[[T_1,T_2,T_3]]$). Paskunas montre \cite[th. 1.4, cor. 4.48]{Pa} que le foncteur de Colmez 
         induit une anti-équivalence entre
          ${\rm Rep}_L(\delta)_{\Pi}$ et la catégorie des 
          $R^{\zeta}$-modules de longueur finie, et cette anti-équivalence est compatible avec ${\rm Ext}^i$. En particulier, on a un 
          isomorphisme canonique 
          $${\rm Ext}^1(\Pi,\Pi)\simeq {\rm Hom}(R^{\zeta}, L[\varepsilon]),$$
          et le terme de droite est de dimension $3$ sur $L$ (le premier ${\rm Ext}^1$ est calculé dans ${\rm Rep}_L(\delta)$). Autrement dit, l'espace ${\rm Ext}^1(\Pi,\Pi)$ des déformations infinitésimales de $\Pi$ dans ${\rm Rep}_L(\delta)$ est isomorphe à l'espace ${\rm Ext}^1_{\zeta}(V(\Pi),V(\Pi))$ des déformations 
 infinitésimales de $V(\Pi)$, de déterminant $\zeta$ en tant que $L[\varepsilon]$-module.
 
     Un rôle important dans les travaux de Paskunas \cite{PaskunasBM} est joué par les déformations 
     $\Pi_1$ de $\Pi$ telles que la suite $0\to \Pi^{\rm alg}\to \Pi_1^{\rm alg}\to \Pi^{\rm alg}\to 0$
     soit exacte. Le point fondamental est que la dimension de cet espace est $\leq 1$ quand 
     $\Pi^{\rm alg}$ est irréductible, ce qui lui permet d'interpréter les anneaux de déformations potentiellement semi-stables
     purement en termes de représentations de ${\rm GL}_2(\qp)$. Paskunas montre dans \cite[th. 4.18, 6.1]{PaskunasBM} ce résultat 
      quand $V(\Pi)$ est trianguline, autrement dit, quand $\Pi$ fait partie de la série
        principale $p$-adique. Sa preuve, qui repose sur une description explicite des vecteurs localement analytiques $\Pi^{\rm an}$ de 
        $\Pi$ (due à Liu, Xie, Zhang \cite{LXZ} et Colmez \cite{Cvectan} et conjecturée par Berger, Breuil \cite{BB} et Emerton \cite{EmCoates}), ainsi que sur les travaux de Berger-Breuil \cite{BB}, ne s'adapte pas au cas où $V(\Pi)$ n'est plus trianguline. Nous allons voir que cela découle du théorème \ref{main} dans tous les cas (si $\Pi^{\rm alg}$ est irréductible).

       Il est clair que si la suite $0\to \Pi^{\rm alg}\to \Pi_1^{\rm alg}\to \Pi^{\rm alg}\to 0$ est exacte, alors $\Pi_1^{\rm alg}$ est dense dans 
     $\Pi_1$, et la réciproque est vraie si en plus $\Pi^{\rm alg}$ est irréductible. Considérons donc le sous-espace
     ${\rm Ext}^1_{\rm alg}(\Pi,\Pi)$
    de ${\rm Ext}^1(\Pi,\Pi)$ engendré par les extensions $0\to \Pi\to\Pi_1\to \Pi\to 0$
      pour lesquelles $\Pi_1^{\rm alg}$ est dense dans $\Pi_1$. Il s'agit donc de l'espace des déformations infinitésimales
      de $\Pi$ qui sont complétions unitaires de leurs vecteurs localement algébriques. 
                              On note ${\rm Ext}^1_{g,\zeta}(V(\Pi),V(\Pi))$
      le sous-espace de ${\rm Ext}^1_{\zeta}(V(\Pi),V(\Pi))$ correspondant aux déformations de $V(\Pi)$ qui sont de de Rham.
       La discussion ci-dessus permet de reformuler le théorème 
      \ref{main} comme suit:
      
      \begin{theorem}
      Soit $p\geq 5$ et soit $\Pi\in {\rm Rep}_L(\delta)$ est un objet supersingulier tel que $\Pi^{\rm alg}\ne 0$. Alors
      on a un isomorphisme canonique, induit par le foncteur de Colmez
       $${\rm Ext}^1_{\rm alg}(\Pi,\Pi)\simeq {\rm Ext}^1_{g,\zeta}(V(\Pi),V(\Pi)).$$
    \end{theorem}

    Il nous reste donc à calculer la dimension de l'espace ${\rm Ext}^1_{g,\zeta}(V(\Pi),V(\Pi))$, ce qui se fait sans mal 
    en utilisant la théorie d'Iwasawa des représentations de de Rham. Le calcul offre quelques surprises. Avant d'énoncer le résultat,
    il nous faut une définition.
   On dit 
      qu'une représentation $V$ absolument irréductible, de dimension $2$ est {\it spéciale} si elle est 
       potentiellement cristalline, à poids de Hodge-Tate distincts
     et si la représentation de Weil-Deligne associée au module de Fontaine $D_{\rm pst}(V)$ est 
      la somme directe de deux caractères, dont le quotient est $|x|$ ou $|x|^{-1}$. Une telle représentation est donc trianguline, car elle
      devient cristalline sur une extension abélienne de $\qp$. Ces représentations peuvent donc être entièrement comprises en termes
      des modules filtrés associés par la théorie de Fontaine.
      
   \begin{proposition} \label{iwas1} Soit $\zeta$ un caractère unitaire et soit $V$ une $L$-représentation absolument irréductible, de de Rham, à poids de Hodge-Tate distincts
  et telle que $\det V=\zeta$. Alors $\dim_L {\rm Ext}^1_{ g, \zeta}(V,V)=1$
   sauf si $V$ est spéciale, auquel cas $\dim_L {\rm Ext}^1_{ g, \zeta}(V,V)=2$.
   \end{proposition} 
   
      Enfin, nous disposons du résultat suivant de Colmez \cite[th. 4.12]{Cvectan}, qui, combiné au théorème \ref{main} et à la proposition \ref{iwas1}
      permet de finir la preuve de 
     \cite[th. 4.18]{PaskunasBM} dans les cas restants
      (i.e. quand la partie lisse de $\Pi^{\rm alg}$ est supercuspidale). 
      
      \begin{proposition} \label{iwas1}
       $\Pi^{\rm alg}$ est réductible si et seulement si $V(\Pi)$ est spéciale. Dans ce cas, $\Pi^{\rm alg}$ vit dans une suite exacte de la forme 
       $$0\to W\otimes_L {\rm St}\to \Pi^{\rm alg}\to W\to 0$$
   pour une certaine représentation algébrique $W$, ${\rm St}$ étant la représentation de Steinberg (lisse).     
      \end{proposition}

        \begin{remark}
         Supposons que $\Pi^{\rm alg}$ est réductible, donc $V(\Pi)$ est spéciale. Soit $\Pi_1$ une déformation 
         de $\Pi^{\rm alg}$ telle que $\Pi_1^{\rm alg}$ soit dense dans $\Pi_1$. Une question naturelle (que je dois à Paskunas)
         est de savoir si dans ce cas la suite $0\to \Pi^{\rm alg}\to \Pi_1^{\rm alg}\to \Pi^{\rm alg}\to 0$ est exacte. 
         Dans \cite{Dthèse} nous "montrons" que la réponse est affirmative, sauf que... la preuve n'est pas correcte. D'ailleurs, il est 
         probable que la réponse est négative, pour la raison suivante. 
        La proposition \ref{iwas1} fournit deux classes
       de déformations de de Rham de $V(\Pi)$: une potentiellement cristalline et l'autre potentiellement semi-stable mais non cristalline.
       Il est probable que le théorème VI.6.30 de Colmez \cite{Cbigone}, reliant le module de Jacquet de la partie lisse de 
       $\Pi^{\rm alg}$ au module de Fontaine 
       $D_{\rm pst}(V(\Pi))$,
       admet une version pour les déformations infinitésimales.
       Si c'était le cas, cela montrerait que la suite $0\to \Pi^{\rm alg}\to \Pi_1^{\rm alg}\to \Pi^{\rm alg}\to 0$
       est exacte si et seulement si $V(\Pi_1)$ est la déformation potentiellement cristalline de $V(\Pi)$.
       On obtiendrait ainsi une suite exacte 
       $0\to \Pi\to\Pi_1\to\Pi\to 0$, avec $\Pi$ supersingulière, 
       $\Pi_1^{\rm alg}$ dense dans $\Pi_1$, mais telle que la suite 
       $0\to \Pi^{\rm alg}\to \Pi_1^{\rm alg}\to \Pi^{\rm alg}\to 0$ ne soit pas exacte. Nous espérons revenir sur ce point dans un futur article.

        \end{remark}

         Paskunas sait démontrer, en utilisant le théorème \ref{main}, ses résultats de \cite{PaskunasBM} et sa "décomposition de Bernstein" de \cite[th. 1.4]{Pa}, la généralisation suivante du
         théorème de Colmez \cite[th. 0.20]{Cbigone}.

     \begin{theorem}\label{genColmez}
         
        Si $p\geq 5$ et si $\Pi\in {\rm Rep}_L(\delta)$ est une complétion de ses vecteurs localement algébriques, alors 
         $V(\Pi)$ est de de Rham.
                 
        \end{theorem}

  La remarque \ref{extalg} montre que la réciproque de l'énoncé ci-dessus est fausse et nous ne savons pas décrire
         seulement en termes galoisiens les objets $\Pi$ de ${\rm Rep}_L(\delta)$ qui sont complétions de leurs vecteurs localement algébriques
         (les techniques de cet article permettent de montrer que la réciproque du théorème \ref{genColmez} est vraie; il reste donc à comprendre
         les représentations non supersingulières et leurs extensions).
         
   \subsection{Ingrédients de la preuve}
      
       Nous terminons cette introduction en expliquant les étapes de la preuve du théorème \ref{main}, qui utilise de manière intensive la théorie
       des $(\varphi,\Gamma)$-modules et les techniques introduites par Colmez pour l'étude de la correspondance de Langlands locale $p$-adique pour ${\rm GL}_2(\qp)$ \cite{Cbigone}. 
       En principe, nous suivons 
       les lignes de la démonstration alternative \cite{Annalen} du théorème de Colmez \cite[th. 0.20]{Cbigone}.
        Cela demande d'étendre/adapter bon nombre de résultats 
       de \cite{Annalen} et \cite{Cbigone}. Nous reposons aussi de manière essentielle sur \cite{CD}, 
       qui généralise (souvent avec des preuves différentes) la plupart des résultats de \cite{Cbigone} à la catégorie ${\rm Rep}_L(\delta)$ (et non seulement aux objets 
       irréductibles). Dans cet article on étend (avec une preuve différente) une loi de réciprocité explicite de Colmez et on raffine sa théorie du modèle de Kirillov
       de $\Pi^{\rm alg}$, ainsi que le résultat principal 
       de \cite{Annalen}. Tout ceci nous est indispensable pour la preuve du théorème \ref{main}.
       
         La plupart de l'article consiste en l'étude fine d'un foncteur $\Pi\to \Pi_c^{P-\rm fini}$ sur ${\rm Rep}_L(\delta)$. L'espace  
         $ \Pi_c^{P-\rm fini}$ est formé des vecteurs $v\in \Pi$ satisfaisant les deux conditions suivantes:
         
         $\bullet$ Il existe $n,k\in\mathbf{N}^*$ et $m\in\mathbf{Z}$ tels que 
         $$\left(\sum_{i=0}^{p^n-1} \left(\begin{smallmatrix} 1 & i \\0 & 1\end{smallmatrix}\right) \right)^k\left(\begin{smallmatrix} p^m & 0 \\0 & 1\end{smallmatrix}\right)v=0.$$
         
         $\bullet$ L'espace $L\left[\left(\begin{smallmatrix} \zpet & 0 \\0 & 1\end{smallmatrix}\right)\right]v$ est de dimension finie sur $L$. 
         
           Supposons dans la suite que tous les facteurs de Jordan-Hölder de $\Pi$ sont supersinguliers. Un résultat remarquable de Colmez \cite{Cbigone} 
           est que tout vecteur de $\Pi_c^{P-\rm fini}$ est en fait localement analytique, et que l'espace $\Pi_c^{P-\rm fini}$ est assez gros
           (dans loc.cit. ceci est démontré pour les objets irréductibles et pour un sous-espace de $\Pi_c^{P-\rm fini}$, mais les arguments 
           restent essentiellement les mêmes dans le cas général). La preuve de ce résultat est très indirecte, malgré la description
           très simple\footnote{Le lecteur pourra essayer de montrer que $\Pi_c^{P-\rm fini}$ est stable sous l'action du groupe de Borel
           directement à partir de la définition de $\Pi_c^{P-\rm fini}$...} de $\Pi_c^{P-\rm fini}$. 
      Elle utilise la description explicite de $\Pi_c^{P-\rm fini}$ en termes du
           $(\varphi,\Gamma)$-module attaché à $\Pi$, la théorie de Sen \cite{Sen} et surtout
           une identité un peu miraculeuse due à Colmez. Cette identité permet de calculer l'action d'un élément de $(\Pi^{\rm an})^*$
           sur un vecteur dans $\Pi_c^{P-\rm fini}$ purement en termes de $(\varphi,\Gamma)$-modules. Nous étendons tous ces résultats de Colmez à notre contexte et nous donnons une nouvelle preuve de cette identité miraculeuse de Colmez. Comme conséquence, nous prouvons le résultat suivant ($B$ étant le Borel supérieur de ${\rm GL}_2(\qp)$)
           
           \begin{theorem}\label{dense}
             Soit $\Pi\in {\rm Rep}_L(\delta)$ dont tous les facteurs de Jordan-Hölder sont supersinguliers. Alors 
              $\Pi_c^{P-\rm fini}$ est un sous-espace dense de $\Pi$, contenu dans $\Pi^{\rm an}$ et stable sous l'action de $B$.
           \end{theorem}
         
             On peut aussi montrer \cite{Dthèse} que si $\Pi$ est supersingulier, alors $\Pi_c^{P-\rm fini}$ est dense dans $\Pi^{\rm an}$
             si et seulement si $V(\Pi)$ n'est pas trianguline (la preuve est nettement plus délicate que celle du théorème \ref{dense}).
             On déduit du théorème \ref{dense} et de la définition de $\Pi_c^{P-\rm fini}$ que
               $\Pi_c^{P-\rm fini}$ est un sous-espace de 
               $$(\Pi^{\rm an})^{U-\rm fini}=\varinjlim_{n}\, (\Pi^{\rm an})^{(u^+)^n=0},$$
               où $u^+$ désigne l'action infinitésimale de $\left(\begin{smallmatrix} 1 &\qp \\0 & 1\end{smallmatrix}\right)$. En particulier,
               $(\Pi^{\rm an})^{U-\rm fini}$ est dense dans $\Pi$ (mais il n'est pas toujours dense dans $\Pi^{\rm an}$).
         
             Pour l'étude des vecteurs localement algébriques, il faut introduire un sous-espace 
             $\Pi_c^{P-\rm alg}$ de $\Pi_c^{P-\rm fini}$ qui est un avatar de $\Pi^{\rm alg}$. Soient $\Pi,\Pi_1$ comme dans le théorème
             \ref{main} et supposons (quitte à faire une torsion par un caractère algébrique) que les poids de 
             Hodge-Tate de $V(\Pi)$ sont $0$ et $k$, avec $k\in\mathbf{N}^*$. On note $\Pi_c^{P-\rm alg}$ 
             le sous-espace de $\Pi_c^{P-\rm fini}$ formé des vecteurs $v$ tels que 
             $$\prod_{i=0}^{k-1} \left (\left(\begin{smallmatrix} 1+p^n & 0 \\0 & 1\end{smallmatrix}\right)-(1+p^n)^{i+1-k}\right)v=0$$
             pour tout $n$ assez grand (il suffit que ce soit vrai pour un seul $n$). Si $v\in\Pi_c^{P-\rm alg}$, alors l'application $g\to (\det g)^{k-1} g\cdot v$ est localement polynômiale
             de degré plus petit que $k$ sur le Borel $B$. 
                Le résultat technique principal, dont la preuve occupe la quasi-totalité de l'article est alors le suivant.
             
             \begin{theorem}\label{main2}
               Soient $\Pi, \Pi_1$ comme dans le théorème \ref{main}.
    
    a) Si $\Pi_1^{\rm alg}\ne \Pi^{\rm alg}$, alors la représentation $V(\Pi_1)$ est de Hodge-Tate.
    
    b) Si $V(\Pi_1)$ est de Hodge-Tate, alors $\Pi_c^{P-\rm alg}$ et $\Pi_{1,c}^{P-\rm alg}$ sont stables sous l'action
    de $B$ et on a une suite exacte scindée de $B$-modules 
    $$0\to \Pi_c^{P-\rm alg}\to \Pi_{1,c}^{P-\rm alg}\to \Pi_c^{P-\rm alg}\to 0.$$
   
   c) $V(\Pi_1)$ est de de Rham si et seulement si $\Pi_{1,c}^{P-\rm alg}\subset \Pi_1^{\rm alg}$, et dans ce 
    cas $\Pi_{1,c}^{P-\rm alg}$ est dense dans $\Pi_1$ (donc $\Pi_1^{\rm alg}$ est aussi dense dans $\Pi_1$).

             \end{theorem}
             
             Il est clair que le théorème \ref{main} est une conséquence du théorème \ref{main2} ci-dessus. Le a) est le contenu du chapitre 3
             et découle d'une analyse fine de l'action infinitésimale de ${\rm GL}_2(\qp)$ sur $\Pi_1^{\rm an}$ (qui généralise le résultat principal
             de \cite{Annalen}).
            Le b) est la réunion des chapitres 4, 5 et 6, et le c) est démontré dans le chapitre 7.

        \subsection{Remerciements} \quad Il sera évident au lecteur combien cet article doit aux travaux de Pierre Colmez et Vytautas Paskunas. Je leur remercie 
                pour m'avoir posé la question dont la réponse fait l'objet de cet article, ainsi que pour les discussions
        très enrichissantes que l'on a eues autour de cette question. Merci à Laurent Berger et à Pierre Colmez pour m'avoir aidé à simplifier la preuve du théorème \ref{iwas}. Une bonne partie de ce papier a été rédigée pendant un séjour au Beijing International Center for Mathematical Research en novembre $2012$, et je voudrais remercier Ruochuan Liu pour l'invitation et les excellentes conditions de travail.

\subsection{Notations} \label{Notations} \quad On fixe une extension finie $L$ de $\qp$ et une suite $(\varepsilon^{(n)})_{n\geq 0}$ de racines $p^n$-ièmes de l'unité, telle que 
$\varepsilon^{(1)}\ne 1$ et $(\varepsilon^{(n+1)})^p=\varepsilon^{(n)}$ pour tout $n\geq 0$. On pose $F_n=\qp(\varepsilon^{(n)})$, $L_n=L\otimes_{\qp} F_n$
et $L_{\infty}=\cup_{n\geq 1} L_n$. On note $\chi: {\rm Gal}(\overline{\qp}/\qp)\to \zpet$ le caractère cyclotomique,
$H={\rm Ker}(\chi)$ et $\Gamma={\rm Gal}(\qp(\mu_{p^{\infty}})/\qp)$. Ainsi, $\chi$ induit un isomorphisme 
$\Gamma\simeq \zpet$, et on note $a\to\sigma_a$ son inverse (on a donc $\sigma_a(\zeta)=\zeta^a$ pour $a\in\zpet$
et $\zeta\in\mu_{p^{\infty}}$).
Si $\delta: \qpet\to O_L^*$ est un caractère unitaire, on note $w(\delta)$ son ${\it poids}$, défini par $w(\delta)=\delta'(1)$, dérivée de $\delta$ en $1$.
Si $\Pi$ est une représentation de Banach de $G$, on note $\Pi^{\rm an}$ le sous-espace des vecteurs localement analytiques de $\Pi$ 
(i.e. les vecteurs $v\in\Pi$ dont l'application orbite $g\mapsto g\cdot v$ est localement analytique sur $G$, à valeurs dans $\Pi$), et $\Pi^{\rm alg}\subset \Pi^{\rm an}$
le sous-espace des vecteurs localement algébriques (l'application orbite est donc localement polynômiale). On renvoie à \cite{Emlocan, STInv, CD} pour l'étude du foncteur
$\Pi\to \Pi^{\rm an}$.

\section{Rappels et complements}

  \quad Ce chapitre est purement préliminaire, le but étant de fixer des notations et de rappeler quelques résultats et constructions que l'on utilisera plus loin.
 
\subsection{$(\varphi,\Gamma)$-modules}

   \quad Soit $\mathcal{R}$ l'anneau de Robba, i.e. l'anneau des séries de Laurent $\sum_{n\in\mathbf{Z}} a_nT^n\in L[[T, T^{-1}]]$ qui convergent sur une couronne
   du type $0< v_p(T)\leq r$ ($r$ dépend de la série). On note $\mathcal{E}^{\dagger}$ le sous-anneau de $\mathcal{R}$ formé des séries bornées et on note 
   $\mathcal{E}$ le complété $p$-adique de $\mathcal{E}^{\dagger}$. Alors $\mathcal{E}^{\dagger}$ et $\mathcal{E}$ sont des sous-corps de $L[[T, T^{-1}]]$ et 
   $\mathcal{E}^{\dagger}\subset \mathcal{E}$. On munit $\mathcal{R}$ d'actions $L$-linéaires continues de $\varphi$ et $\Gamma$ en posant (pour $a\in\zpet$)
   $$\varphi(T)=(1+T)^p-1, \quad \sigma_a(T)=(1+T)^a-1.$$ Ces actions laissent stable $\mathcal{E}^{\dagger}$
et se prolongent par continuité en des actions sur $\mathcal{E}$. On dispose d'un inverse à gauche $\psi$ de $\varphi$
sur ces anneaux, qui commute à l'action de $\Gamma$. 

    Soit $\Phi\Gamma^{\rm et}(\mathcal{E})$ la catégorie des $(\varphi,\Gamma)$-modules étales sur $\mathcal{E}$. Ses objets sont les $\mathcal{E}$-espaces vectoriels
    $D$ de dimension finie, munis d'actions semi-linéaires 
   continues de $\varphi$ et de $\Gamma$, qui commutent et telles que l'action de $\varphi$ soit étale\footnote{Cela veut dire que $D$ a une base dans laquelle la matrice 
   de $\varphi$ est dans ${\rm GL}_d(O_{\mathcal{E}})$, $O_{\mathcal{E}}$ étant l'anneau des entiers de $\mathcal{E}$.}. La catégorie $\Phi\Gamma^{\rm et}(\mathcal{E})$ est équivalente
   à la catégorie ${\rm Rep}_L(G_{\qp})$ des $L$-représentations continues, de dimension finie de ${\rm Gal}(\overline{\qp}/\qp)$, d'après un théorème de Fontaine \cite{FoGrot}. On notera $V(D)$ (resp. $D(V)$)
   la représentation galoisienne (resp. le $(\varphi,\Gamma)$-module étale) associée à $D$ (resp. $V$). 
   
   \begin{definition}
    Le dual de Cartier $\check{D}$ d'un $(\varphi,\Gamma)$-module étale $D$ est 
     $\check{D}=D(V(D)^*\otimes \chi)$, où $V(D)^*$ désigne le $L$-dual de $V(D)$.
   \end{definition}
  
   Pour tout $D\in\Phi\Gamma^{\rm et}(\mathcal{E})$ il existe un plus grand sous-$\mathcal{E}^{\dagger}$-espace vectoriel $D^{\dagger}$ qui est stable par 
      $\varphi$. De plus, $D^{\dagger}$ est stable par $\Gamma$ et $D=\mathcal{E}\otimes_{\mathcal{E}^{\dagger}} D^{\dagger}$ d'après \cite{CCsurconv}. On pose
      $D_{\rm rig}=\mathcal{R}\otimes_{\mathcal{E}^{\dagger}} D^{\dagger}$, un $(\varphi,\Gamma)$-module de pente nulle sur 
      $\mathcal{R}$.

\subsection{Le dictionnaire ${\rm Gal}(\overline{\qp}/\qp)-{\rm GL}_2(\qp)$}

   \quad On fixe un caractère unitaire $\delta:\qpet\to O_L^*$ et on renvoie au premier paragraphe de l'introduction pour la définition de la catégorie 
   ${\rm Rep}_L(\delta)$. On dispose \cite[ch. IV]{Cbigone} du foncteur de Colmez 
   $${\rm Rep}_L(\delta)\to \Phi\Gamma^{\rm et}(\mathcal{E}), \quad \Pi\mapsto D(\Pi),$$
   qui est contravariant et exact. On note
  $\mathcal{C}_L(\delta)$ son image essentielle. 
  
    Dans l'autre sens, pour tout $D\in \Phi\Gamma^{\rm et}(\mathcal{E})$ 
   on définit \cite[ch. II]{Cbigone} un 
  faisceau $G$-équivariant\footnote{L'action de $G$ sur $\mathbf{P}^1(\qp)$ est celle usuelle, par des homographies.} $U\to D\boxtimes_{\delta} U$ sur $\p1(\qp)$, dont les sections sur $\zp$ sont $D$. Si $U$ est un ouvert compact de $\qp$, on dispose d'une application d'extension par $0$ qui permet de voir 
  $D\boxtimes_{\delta} U$ comme sous-module de $D\boxtimes_{\delta}\p1$. 
Puisque $\mathbf{P}^1(\qp)=\zp\cup  \left(\begin{smallmatrix} 0 & 1 \\1 & 0\end{smallmatrix}\right)\zp$, l'espace des sections globales $D\boxtimes_{\delta}\p1$ de ce faisceau s'injecte naturellement dans $D\times D$, 
ce qui permet de le munir de la topologie induite, $D$ étant muni de la topologie faible. En tant que $L$-espace vectoriel topologique, $D\boxtimes_{\delta}\p1$ est une extension d'un $L$-espace de Banach par le dual d'un $L$-espace de Banach. En effet, on a un isomorphisme d'espaces topologiques $D\boxtimes_{\delta}\p1\simeq D\times D\simeq \mathcal{E}^{2\dim_{\mathcal{E}}(D)}$ (le premier étant induit par $z\mapsto ({\rm Res}_{\zp}(z), {\rm Res}_{\zp}  \left(\begin{smallmatrix} 0 & 1 \\1 & 0\end{smallmatrix}\right) z)$, le second
par le choix d'une base de $D$), ainsi qu'une suite exacte fondamentale de Colmez \cite[remarque 0.2]{Cbigone}
$$0\to \mathcal{D}_0(\zp, L)\to \mathcal{E}\to \mathcal{C}(\zp,L)\to 0,$$
où $\mathcal{D}_0(\zp,L)$ est l'espace des mesures sur $\zp$ à valeurs dans $L$ (dual faible 
du $L$-Banach $\mathcal{C}(\zp,L)$).
En général, il n'existe pas de telle décomposition de $D\boxtimes_{\delta}\p1$ qui soit en plus $G$-équivariante, mais 
le théorème suivant montre que c'est le cas si $D\in \mathcal{C}(\delta)$ (le lecteur pourra vérifier que la réciproque est vraie).

  \begin{theorem}\label{CD}  a) Si $D\in\mathcal{C}_L(\delta)$, alors $\check{D}\in \mathcal{C}_L(\delta^{-1})$. 
  
  b) Il existe un foncteur covariant $$\mathcal{C}_L(\delta)\to {\rm Rep}_L(\delta), \quad D\to \Pi_{\delta}(D)$$
     tel que pour tout $D\in \mathcal{C}_L(\delta)$ on ait une suite exacte de $G$-modules topologiques
     $$0\to \Pi_{\delta^{-1}}(\check{D})^*\to D\boxtimes_{\delta}\p1\to \Pi_{\delta}(D)\to 0.$$
     
  c) On a $D(\Pi_{\delta}(D))=\check{D}$ pour tout 
  $D\in \Phi\Gamma(\delta)$. 
  
  d) Si tous les facteurs de Jordan-Hölder de $\Pi$ sont supersinguliers\footnote{Rappelons qu'un objet $\Pi$ de ${\rm Rep}_L(\delta)$ 
  est dit supersingulier si $\Pi$ est absolument irréductible et pas isomorphe à un sous-quotient de l'induite parabolique d'un caractère unitaire.}, alors on a un isomorphisme canonique
  $\Pi_{\delta}(\check{D}(\Pi))\simeq \Pi$.  
  
  \begin{proof}
  C'est la réunion des théorèmes III.4, III.48,  de la remarque III.31 et du corollaire III.46 de \cite{CD}.
  
    \end{proof}

    \end{theorem}
    
    Pour tout $D\in\mathcal{C}_L(\delta)$ on montre \cite[ch. VI]{CD} que le faisceau $U\to D\boxtimes_{\delta} U$ induit un faisceau 
    $G$-équivariant $U\to D_{\rm rig}\boxtimes_{\delta} U$ sur $\p1(\qp)$, dont les sections sur $\zp$ sont 
    $D_{\rm rig}$. 
    
    \begin{theorem}\label{analitic}
     Pour tout $D\in\mathcal{C}_L(\delta)$ il existe une suite exacte de $G$-modules topologiques 
     $$0\to (\Pi_{\delta^{-1}}(\check{D})^{\rm an})^*\to D_{\rm rig}\boxtimes_{\delta}\p1\to \Pi_{\delta}(D)^{\rm an}\to 0.$$
    De plus, $D_{\rm rig}\boxtimes_{\delta}\p1$ est un module topologique sur l'algèbre des distributions de ${\rm GL}_2(\zp)$ et la suite exacte précédente
    est une suite exacte de modules topologiques sur cette algèbre.
    \end{theorem}
 
 \begin{proof}
  C'est la réunion de la proposition VI.7 et du corollaire VI.12 de \cite{CD}.
 
 \end{proof}

\subsection{Anneaux de périodes $p$-adiques}\label{Font}

  \quad On renvoie à \cite{FoAnnals} et \cite{CvectdR} pour les preuves des résultats énoncés dans ce $\S$. On note $\mathbf{C}_p$ le compl\'{e}t\'{e} de $\overline{\qp}$, $O_{\mathbf{C}_p}$ l'anneau de ses entiers et $\mathbf{\tilde{E}}^+=\varprojlim_{x\mapsto x^p} O_{\mathbf{C}_p}/p$.
  Alors $\mathbf{\tilde{E}}^+$ est un anneau 
  parfait de caractéristique $p$, et on note $\mathbf{\tilde{E}}$ son corps des fractions. 
  Les anneaux de Fontaine $\mathbf{\tilde{A}}^+$ et $\mathbf{\tilde{A}}$ sont les anneaux des vecteurs de Witt 
  à coefficients dans $\mathbf{\tilde{E}}^+$ et $\mathbf{\tilde{E}}$. 
  On pose $\tilde{\mathbf{B}}=\tilde{\mathbf{A}}[\frac{1}{p}]$
  et $\tilde{\mathbf{B}}^+=\tilde{\mathbf{A}}^+[\frac{1}{p}]$. 
   Les anneaux $\mathbf{\tilde{A}}^+, \mathbf{\tilde{A}}, \tilde{\mathbf{B}}^+, \tilde{\mathbf{B}}$ sont munis 
  d'un Frobenius bijectif $\varphi$ et d'une action de ${\rm Gal}(\overline{\qp}/\qp)$, qui commutent. Le choix (voir le $\S$\ref{Notations}) des racines $p^n$-ièmes 
  de l'unité $\varepsilon^{(n)}$ permet de définir\footnote{$1+T$ est le représentant de Teichmuller de l'élément 
  $(\varepsilon^{(0)}\pmod p, \varepsilon^{(1)}\pmod p,...)$ de $\mathbf{\tilde{E}}^+$.} 
   un élément $T$ de  $(\tilde{\mathbf{A}}^+)^H$, tel que 
  $\varphi(T)=(1+T)^p-1$ et $\sigma_a(T)=(1+T)^a-1$ pour $a\in\zpet$.
   Le séparé complété $\mathbf{B}_{\rm dR}^+$ de $\mathbf{\tilde{B}}^+$ pour la topologie $\frac{T}{\varphi^{-1}(T)}$-adique est 
    un anneau de valuation discrète, d'uniformisante $\frac{T}{\varphi^{-1}(T)}$ ou $t=\log(1+T)$, et de corps résiduel 
    $\mathbf{C}_p$. On note $\mathbf{B}_{\rm dR}= \mathbf{B}_{\rm dR}^+[\frac{1}{t}]$ le corps des fractions de $\mathbf{B}_{\rm dR}^+$. 

   L'inclusion $\zp[T]\subset \mathbf{\tilde{A}}^+$ s'étend en une inclusion $\zp[[T]]\subset \mathbf{\tilde{A}}^+$, ce qui permet de définir, pour 
   $b\in \qp$ et $n\geq -v_p(b)$, un élément 
  $$[(1+T)^b]=\varphi^{-n}\left((1+T)^{p^nb}\right)\in (\mathbf{\tilde{A}}^+)^H,$$
  qui ne dépend pas du choix de $n$. On définit aussi $\varepsilon(b)=(\varepsilon^{(n)})^{p^n b}$ pour 
  $n\geq -v_p(b)$, obtenant ainsi un caractère localement constant $\varepsilon: \qpet\to \mu_{p^{\infty}}$, tel que 
    $$[(1+T)^b]=\varepsilon(b)\cdot e^{tb}\in (\mathbf{B}_{\rm dR}^+)^H.$$ Le lemme suivant est standard et nous l'utiliserons plusieurs fois dans la suite. 
Pour le confort du lecteur, on en donne une preuve.

\begin{lemma}\label{frobdr}
 a) Soit $a\in\mathbf{Z}$. Alors $\varphi^a(T)\in t\mathbf{B}_{\rm dR}^+$ si et seulement si
 $a\geq 0$. Dans ce cas, on a $\frac{\varphi^a(T)}{\varphi^b(T)}\in \mathbf{B}_{\rm dR}^+$ pour tout
 $b\in\mathbf{Z}$.

 b) Si $j\notin [1,n]$, alors $\varphi^{-j}\left(\frac{T}{\varphi^n(T)}\right)\in \mathbf{B}_{\rm dR}^+$.

 c) Soit $x\in \tilde{\mathbf{A}}^+$ et $k\geq 1$. Alors $x\in T^k\tilde{\mathbf{A}}^+$ si et seulement si
 $\varphi^n(x)\in t^k \mathbf{B}_{\rm dR}^+$ pour tout $n\geq 0$.

\end{lemma}

\begin{proof}
Le point a) découle du fait que $\varphi^n(T)=e^{p^nt}-1\equiv
p^nt\pmod{t^2\mathbf{B}_{\rm dR}^+}$ pour tout $n\geq 0$ et $$\varphi^{-n}(T)=\varepsilon^{(n)}e^{t/p^n}-1\equiv \varepsilon^{(n)}-1\ne 0\pmod {t\mathbf{B}_{\rm dR}^+},$$
  pour $n\geq 1$.
Le b) est une cons\'{e}quence imm\'{e}diate du a). En ce qui concerne c), un sens \'{e}tant \'{e}vident, supposons que  $\varphi^n(x)\in t^k \mathbf{B}_{\rm dR}^+$ pour tout $n\geq 0$ et montrons que $x\in T^k\tilde{\mathbf{A}}^+$. Pour $k=1$, c'est un r\'{e}sultat standard \cite[lemme III.3.7]{Cannals}. Si $k>1$, le cas $k=1$ nous dit
qu'il existe $y\in \tilde{\mathbf{A}}^+$ tel que $x=Ty$. Alors $\varphi^n(y)\in t^{k-1}\frac{t}{\varphi^n(T)} \mathbf{B}_{\rm dR}^+=
t^{k-1}\mathbf{B}_{\rm dR}^+$ (par la preuve de a)), ce qui permet de conclure par
r\'{e}currence sur $k$.

\end{proof}

  \subsection{Sous-modules de $\mathbf{B}_{\rm dR}^+\otimes_{\qp} V(D)$ et connexions}\label{con}
  
\quad   Pour tout $n\geq 1$, on note $r_n=\frac{1}{p^{n-1}(p-1)}$ et on note $\mathcal{E}^{]0,r_n]}$ l'anneau des fonctions analytiques sur la couronne $0<v_p(T)\leq r_n$, définies
   sur $L$. Si $D\in \Phi\Gamma^{\rm et}(\mathcal{E})$, alors on peut trouver $m(D)$ assez grand et, pour tout $n\geq m(D)$, un unique 
  sous $\mathcal{E}^{]0,r_n]}$-module $D^{]0,r_n]}$ de $D_{\rm rig}$ tel que
  
  $\bullet$ $D_{\rm rig}=\mathcal{R}\otimes_{\mathcal{E}^{]0,r_n]}} D^{]0,r_n]}$ et 
  
  $\bullet$ $\mathcal{E}^{]0,r_{n+1}]}\otimes_{\mathcal{E}^{]0,r_n]}} D^{]0,r_n]}$ possède une base contenue dans $\varphi(D^{]0,r_n]})$. 
 
  Le module $D^{]0,r_n]}$ est stable sous l'action de $\Gamma$ et $\varphi(D^{]0,r_n]})\subset D^{]0,r_{n+1}]}$ pour $n\geq m(D)$ (voir 
 \cite[th. I.3.3]{BerAst} pour ces résultats). 
  
     Rappelons que $H={\rm Ker}(\chi)$ et posons $$\tilde{D}_{\rm dif}=(\mathbf{B}_{\rm dR}\otimes_{\qp} V(D))^H, \quad \tilde{D}_{\rm dif}^+=(\mathbf{B}_{\rm dR}^+\otimes_{\qp} V(D))^{H}.$$ 
    On dispose \cite[def. 4]{Annalen} d'un morphisme $\varphi^{-n}: D^{]0,r_n]}\to \tilde{D}_{\rm dif}^+$ de localisation en $\varepsilon^{(n)}-1$, qui est 
     $L[\Gamma]$-linéaire et injectif. On note $D_{{\rm dif}, n}^+$ le sous $L_n[[t]]$-module de $\tilde{D}^+_{\rm dif}$ engendré
     par $\varphi^{-n}(D^{]0,r_n]})$. C'est un $L_n[[t]]$-module libre de rang $\dim_{\mathcal{E}}(D)$ si $n\geq m(D)$, et il est stable sous l'action de $\Gamma$. Enfin, on note 
     $D_{{\rm dif},n}=L_n((t))\otimes_{L_n[[t]]} D_{{\rm dif}, n}^+$.  Si $V(D)$ est de de Rham, alors on a un isomorphisme canonique $D_{{\rm dif},n}\simeq L_n((t))\otimes_{L} D_{\rm dR}(V)$, où 
     $D_{\rm dR}(V)=(\mathbf{B}_{\rm dR}\otimes_{\qp} V)^{{\rm Gal}(\overline{\qp}/\qp)}$.
     
      Un résultat standard de Berger \cite[lemme 4.1]{Ber} montre que l'action de $\Gamma$ sur $D_{\rm rig}$ se dérive, d'où une connexion 
     $$\nabla: D_{\rm rig}\to D_{\rm rig}, \quad \nabla(z)=\lim_{a\to 1} \frac{\sigma_a(z)-z}{a-1},$$
     qui préserve $D^{]0,r_n]}$ pour $n\geq m(D)$. L'action de $\Gamma$ sur $D_{{\rm dif},n}^+$ (avec $n\geq m(D)$) se dérive aussi, 
     d'où une connexion $L_n[\Gamma]$-linéaire $\nabla$ sur $D_{{\rm dif},n}^+$, au-dessus de la connexion $t\frac{d}{dt}$ sur 
     $L_n[[t]]$. 
     Cette connexion $\nabla$ induit un endomorphisme du $L_n$-module
    libre de type fini $D_{{\rm Sen},n}:=D_{{\rm dif},n}^+/tD_{{\rm dif},n}^+$. 
    Les valeurs propres (resp. le polyn\^{o}me caract\'{e}ristique) de cet endomorphisme
    sont  les poids de Hodge-Tate
  g\'{e}n\'{e}ralis\'{e}s (resp. le polyn\^{o}me $P_{\rm Sen,D}$ de Sen) de
   $D$. On note $\Theta_{{\rm Sen},D}$ l'op\'{e}rateur
  de Sen, i.e. l'op\'{e}rateur $\nabla\pmod t$ agissant sur la r\'{e}union des $D_{{\rm Sen},n}$. 
  Nous aurons besoin de l'extension suivante de \cite[prop. 2]{Annalen}:  
  
\begin{proposition}\label{difpol}
 Si $P\in L[X]$, alors $P(\Theta_{{\rm Sen}, D})=0$
 \'{e}quivaut \`{a} $P(\nabla)(D_{\rm rig})\subset t\cdot D_{\rm rig}$.

\end{proposition}

\begin{proof}   Si $z\in D_{\rm rig}$ et $\varphi^{-n}(z)\in t D_{{\rm dif}, n}^+$ pour
    tout $n$ assez grand, alors $z\in t D_{\rm rig}$ d'après \cite[lemmes 5.1, 5.4]{Ber}. Ensuite, $D_{\rm rig}$ est la r\'{e}union des
    $D^{]0,r_n]}$, qui sont stables par $\nabla$ pour $n$ assez grand et satisfont $D^{]0,r_n]}\cap tD_{\rm rig}=tD^{]0,r_n]}$. En utilisant tout ceci et le fait que $\varphi^{-n}$ commute \`{a} $\nabla$, on obtient que
    $P(\nabla)(D_{\rm rig})\subset t\cdot D_{\rm rig}$ si et seulement si
    $P(\nabla)(\varphi^{-n}(D^{]0,r_n]}))\subset t\cdot D_{{\rm dif},n}^+$ pour
    tout $n$ assez grand. Cela \'{e}quivaut \`{a} $P(\nabla)(D_{{\rm dif},n}^+)\subset
    t\cdot D_{{\rm dif},n}^+$ pour $n$ assez grand, et arrive si et seulement si
    $P(\Theta_{{\rm Sen}, D})=0$, ce qui permet de conclure.

\end{proof}

  \section{Extensions de Hodge-Tate}
  
     \quad Soit $D\in \Phi\Gamma^{\rm et}(\mathcal{E})$ de dimension $2$, tel que $V(D)$ soit absolument irréductible, de de Rham, à poids de Hodge-Tate
     $1$ et $1-k$, avec $k\in\mathbf{N}^*$. Ainsi, les poids de Hodge-Tate de $\check{D}$ sont $0$ et $k$. On note $\delta=\chi^{-1}\det V(D)$, que l'on voit 
     comme caractère unitaire de $\qpet$ par la théorie locale du corps de classes. Notons que $w(\delta)=1-k$. D'après 
  le théorème fondamental \cite[th. II.3.1]{Cbigone} de la correspondance de Langlands locale $p$-adique pour ${\rm GL}_2(\qp)$, on a $D\in \mathcal{C}_L(\delta)$.
   De plus, $\Pi_{\delta}(D)^{\rm alg}\ne 0$, d'après un théorème profond de Colmez \cite[th. 0.20]{Cbigone}.
   
     On fixe une extension non scindée 
     $0\to D\to E\to D\to 0$ telle que $E\in \mathcal{C}_{L}(\delta)$. La suite $0\to \Pi_{\delta}(D)\to \Pi_{\delta}(E)\to \Pi_{\delta}(D)\to 0$ est 
     exacte\footnote{Attention au fait que le foncteur 
     $D\to \Pi_{\delta}(D)$ n'est pas exact dans $\mathcal{C}_L(\delta)$; il faut quotienter 
     $\mathcal{C}_L(\delta)$ par la sous-catégorie des représentations de dimension finie pour obtenir un foncteur exact.} \cite[remarque III.31]{CD} dans ${\rm Rep}_L(\delta)$.
      Nous allons montrer que si $\Pi_{\delta}(E)^{\rm alg}\ne \Pi_{\delta}(D)^{\rm alg}$, alors $E$ est de Hodge-Tate
     et l'action de l'algèbre enveloppante $U(\mathfrak{gl}_2)$ de $\mathfrak{gl}_2={\rm Lie}({\rm GL}_2(\qp))$ sur $\Pi_{\delta}(E)^{\rm an}$ est particulièrement simple.       
     Ceci jouera un rôle fondamental dans les chapitres suivants.
     
\subsection{L'élément de Casimir}

  \quad Si $\mathcal{D}\in \mathcal{C}_L(\delta)$, le théorème \ref{analitic} munit $\mathcal{D}_{\rm rig}\boxtimes_{\delta}\p1$ d'une action naturelle de
$U(\mathfrak{gl}_2)$, qui laisse stable $\mathcal{D}_{\rm rig}=\mathcal{D}_{\rm rig}\boxtimes_{\delta} \zp\subset \mathcal{D}_{\rm rig}\boxtimes_{\delta}\p1$ (ce dernier point découle de \cite[prop. VI.8]{CD}). On note 
    $$h=\left(\begin{smallmatrix} 1 & 0 \\0 & -1\end{smallmatrix}\right),\quad
   u^+=\left(\begin{smallmatrix} 0 & 1 \\0 & 0\end{smallmatrix}\right), \quad u^-=\left(\begin{smallmatrix} 0 & 0 \\1 & 0\end{smallmatrix}\right)\in\mathfrak{gl}_2$$
   et on considère l'élément de Casimir 
       $$C=u^+u^-+u^-u^++\frac{1}{2}h^2\in U(\mathfrak{gl}_2).$$ Alors $C$ induit 
    un endomorphisme du $L$-espace vectoriel topologique $\mathcal{D}_{\rm rig}$.
    
\begin{theorem}\label{Casimir}

a) L'opérateur $C$ sur $\mathcal{D}_{\rm rig}$ est $\mathcal{R}$-lin\'{e}aire et commute avec $\varphi$ et $\Gamma$.

b) On a une \'{e}galit\'{e} d'op\'{e}rateurs sur $\mathcal{D}_{\rm rig}$
  $$C=2tu^-+\frac{(2\nabla-(w(\delta)+1))^2-1}{2}.$$

\end{theorem}

\begin{proof} Pour le point a), voir le premier paragraphe de la preuve du théorème 6 de \cite{Annalen} (il y a des hypothèses supplémentaires dans loc.cit, mais 
la preuve marche telle quelle dans ce degré de généralité puisqu'elle n'utilise que le fait que $C$ commute à l'action adjointe de $G$). 
Le b) se démontre comme le lemme 2 de loc.cit, la seule différence étant le calcul de l'action de $h$. Comme celui-ci ne pose aucun problème (voir la preuve de la prop. 4 de loc.cit), cela permet de conclure.
\end{proof}

  \begin{remark}
   Soit $\Pi\in {\rm Rep}_L(\delta)$ un objet supersingulier. Le corollaire III.46 de \cite{CD} montre que 
 le dual de Cartier $\mathcal{D}$ de $D(\Pi)$ est absolument irréductible, de dimension au moins $2$, et $\Pi\simeq \Pi_{\delta}(\mathcal{D})$. On en déduit
   que ${\rm End}_{\varphi,\Gamma,\mathcal{R}}(\mathcal{D}_{\rm rig})=L$ et donc 
   $C$ agit par une constante $c\in L$ sur $\mathcal{D}_{\rm rig}$. Le théorème précédent
   montre que $\frac{(2\nabla-(w(\delta)+1))^2-1}{2}-c$ envoie $\mathcal{D}_{\rm rig}$ dans 
   $t\cdot\mathcal{D}_{\rm rig}$ et la proposition \ref{difpol} permet de conclure que 
   $$(2\Theta_{{\rm Sen},\mathcal{D}}-(1+w(\delta)))^2=2c+1.$$
    Cela montre en particulier que $V(\mathcal{D})$ a au plus deux poids de Hodge-Tate distincts. 
    Bien sûr, si $p\geq 5$, on sait grâce aux travaux de Paskunas que $\dim \mathcal{D}=2$, mais
    ce que l'on vient de faire fonctionne aussi pour $p=2$ et $p=3$. D'ailleurs, on peut utiliser \cite{Dext}
    aussi cette observation pour calculer des extensions dans ${\rm Rep}_L(\delta)$. 
  \end{remark}
     
 \subsection{Caract\'{e}risation des extensions de Hodge-Tate}
 
  \quad On revient aux hypothèses du début de ce chapitre. Il découle de la proposition \ref{difpol} (appliquée à $\check{D}$ et à son polynôme de Sen $X(X-k)$) et du théorème \ref{Casimir}
  que l'élément de Casimir $C$ agit par multiplication par $\frac{k^2-1}{2}$ sur $\check{D}_{\rm rig}\boxtimes_{\delta^{-1}}\p1$
  et que pour tout $z\in\check{D}_{\rm rig}$ on a 
   $$u^{-}(z)=\frac{\nabla(k-\nabla)(z)}{t}.$$
        
        Le résultat principal de ce chapitre est le suivant:

  \begin{theorem}\label{EHT}
     Les assertions suivantes sont \'{e}quivalentes:

  a) $\check{E}$ est de Hodge-Tate;

  b)  On a $\Theta_{{\rm Sen}, \check{E}}^2=k\cdot \Theta_{{\rm Sen},\check{E}}$;

  c) L'\'{e}l\'{e}ment de Casimir agit par multiplication par $\frac{k^2-1}{2}$ sur $\check{E}_{\rm rig}\boxtimes_{\delta^{-1}}\p1$.

  d) Pour tout
   $z\in \check{E}_{\rm rig}$ on a $$u^{-}(z)=\frac{\nabla(k-\nabla)(z)}{t}.$$

  \end{theorem}

\begin{proof}  Comme $\check{E}$ est
de Hodge-Tate si et seulement si $\Theta_{{\rm Sen},\check{E}}$ est diagonalisable
\`{a} valeurs propres enti\`{e}res et comme le polyn\^{o}me de Sen de $\check{E}$ est
$X^2(X-k)^2$, il est clair que a) \'{e}quivaut \`{a} b). L'\'{e}quivalence
entre c) et d) est une cons\'{e}quence imm\'{e}diate du théorème \ref{Casimir}
appliqu\'{e} \`{a} $\check{E}_{\rm rig}$ et $\delta^{-1}$ (noter que $w(\delta^{-1})=k-1$). Si d) est vraie, alors
en posant $P(X)=X(k-X)$, on a $P(\nabla)(\check{E}_{\rm rig})
\subset t\check{E}_{\rm rig}$, puisque $u^{-}$ laisse stable $\check{E}_{\rm rig}$. En utilisant la 
 proposition \ref{difpol}, on obtient $\Theta_{{\rm Sen}, \check{E}}^2=k\cdot \Theta_{{\rm Sen},\check{E}}$, ce qui montre
 que d) entraîne b). 

   Supposons maintenant que b) est vraie et montrons le c). La suite exacte $0\to \check{D}_{\rm rig}\to \check{E}_{\rm rig}\to \check{D}_{\rm rig}\to 0$ nous permet de voir $\check{E}_{\rm rig}$ comme une d\'{e}formation de $\check{D}_{\rm rig}$ \`{a} $L[\varepsilon]=L[X]/X^2$. Comme cette suite
exacte n'est pas scind\'{e}e (sinon la suite $0\to D\to E\to D\to 0$ serait aussi scind\'{e}e) et comme $D$ (et donc $\check{D}$) est absolument irr\'{e}ductible, on obtient
${\rm End}_{\varphi,\Gamma,\mathcal{R}}(\check{E}_{\rm rig})=L[\varepsilon]$. 
On en déduit, en utilisant le théorème \ref{Casimir}, l'existence 
de $a,b\in L$ tels que $Cz=(a+b\varepsilon) z$ pour tout $z\in \check{E}_{\rm rig}$. Puisque 
 $w(\delta^{-1})=k-1$, la partie b) du théorème \ref{Casimir} montre que l'on a une \'{e}galit\'{e} d'op\'{e}rateurs
 $$C-\frac{k^2-1}{2}=2tu^{-}+\frac{(2\nabla-k)^2-k^2}{2}$$ sur $\check{E}_{\rm rig}$.
 Or la proposition \ref{difpol} et l'hypoth\`{e}se que b) est satisfaite entra\^{i}nent l'inclusion $$\frac{(2\nabla-k)^2-k^2}{2}(\check{E}_{\rm rig})\subset t\cdot \check{E}_{\rm rig}.$$ Ainsi, $(a-\frac{k^2-1}{2})z+b\varepsilon(z)\in t\cdot
 \check{E}_{\rm rig}$ pour tout $z\in\check{E}_{\rm rig}$. En prenant $z\in {\rm Ker}(\varepsilon)=\check{D}_{\rm rig}
 \subset \check{E}_{\rm rig}$ on obtient $(a-\frac{k^2-1}{2})\check{D}_{\rm rig}\subset
 t\cdot \check{D}_{\rm rig}$, donc $a=\frac{k^2-1}{2}$. On en d\'{e}duit que
 $b\cdot \varepsilon(\check{E}_{\rm rig})\subset t\cdot \check{E}_{\rm rig}$ et comme $\varepsilon(\check{E}_{\rm rig})=\check{D}_{\rm rig}$, on a finalement $b=0$
 et $C=\frac{k^2-1}{2}$. Cela permet de conclure.
 \end{proof}
 
 Posons $$\nabla_{2k}=\prod_{i=0}^{2k-1}(\nabla-i).$$

\begin{corollary}
  \label{long}
     Si $E$ est de Hodge-Tate et si $\mathcal{D}\in \{D, E\}$, alors 
   $$(u^{-})^k(\check{z})=(-1)^k\frac{\nabla_{2k}(\check{z})}{t^k}$$
 pour tout $\check{z}\in \check{\mathcal{D}}_{\rm rig}$. 
  \end{corollary}

\begin{proof} Le théorème \ref{EHT} et la discussion qui le précède montrent que 
$u^{-}(\check{z})=\frac{\nabla(k-\nabla)(\check{z})}{t}$ pour tout 
$\check{z}\in \check{\mathcal{D}}_{\rm rig}$. Une récurrence immédiate permet d'en déduire la relation
    $$(u^-)^j(\check{z})=\frac{\nabla(\nabla-1)...(\nabla-j+1)(k-\nabla)(k+1-\nabla)...(k+j-1-\nabla)(\check{z})}{t^j},$$
    pour tout $j\geq 1$ et $\check{z}\in \check{\mathcal{D}}_{\rm rig}$. On conclut en prenant $j=k$.
\end{proof}

 \subsection{Application à l'étude de $\Pi_{\delta}(E)^{\rm alg}$}
 
 \quad Si $k\geq 1$ et $l$ sont des entiers, on note
    $W_{l,k}={\rm Sym}^{k-1}(L^2)\otimes \det^l$, o\`{u} ${\rm Sym}^{k-1}(L^2)$ est la puissance sym\'{e}trique $(k-1)$-i\`{e}me de la repr\'{e}sentation
    standard de $G$ sur $L\oplus L$. 
    
    \begin{proposition} \label{pilc}
    
    Il existe des repr\'{e}sentations lisses admissibles $\Pi_{\delta}(D)^{\rm lc}$ et 
    $\Pi_{\delta}(E)^{\rm lc}$ telles que $\Pi_{\delta}(D)^{{\rm alg}}=W_{1-k,k}\otimes\Pi_{\delta}(D)^{\rm lc}$ et 
     $\Pi_{\delta}(E)^{\rm alg}=W_{1-k,k}\otimes
   \Pi_{\delta}(E)^{\rm lc}$.    
    \end{proposition}

\begin{proof}
 L'assertion concernant $\Pi_{\delta}(D)^{{\rm alg}}$ découle de \cite[prop. 11]{Annalen}, en tordant par $\chi^{k-1}$. La preuve de loc.cit. (ou encore 
\cite[lemma 4.9]{PaskunasBM}) montre que $\Pi_{\delta}(D)^{\rm alg}$ ne contient aucune repr\'{e}sentation du type
 $W_{l',k'}\otimes \pi$, avec $\pi$ lisse et $(l',k')\ne (1-k,k)$. En utilisant ceci, la suite exacte $0\to \Pi_{\delta}(D)^{\rm alg}\to \Pi_{\delta}(E)^{\rm alg}\to \Pi_{\delta}(D)^{\rm alg}$
 et \cite[prop. 10]{Annalen}, 
 le résultat s'ensuit.
\end{proof}

\begin{remark}
  L'existence de $\Pi_{\delta}(D)^{\rm lc}$ est un résultat de Colmez \cite[prop. VI.5.1]{Cbigone}, qu'il démontre
  en utilisant sa théorie du modèle de Kirillov de $\Pi_{\delta}(D)^{\rm alg}$. La preuve de \cite{Annalen} n'utilise que les caractères centraux et infinitésimaux de 
  $\Pi_{\delta}(D)^{\rm an}$.
\end{remark}

 \begin{proposition}\label{HTalg}
  Si $\Pi_{\delta}(E)^{\rm alg}\ne \Pi_{\delta}(D)^{\rm alg}$, alors $E$ est de Hodge-Tate.
 \end{proposition}

 \begin{proof} Notons pour simplifier $\Pi=\Pi_{\delta}(D)$ et 
 $\Pi_1=\Pi_{\delta}(E)$.
 La suite exacte $0\to D\to E\to D\to 0$ induit une suite exacte 
 $0\to \Pi\to \Pi_1\to\Pi\to 0$ et donc
 (par exactitude \cite[th. 7.1]{STInv} du foncteur $\Pi\to \Pi^{\rm an}$) une suite
 exacte $0\to \Pi^{\rm an}\to \Pi_1^{\rm an}\to\Pi^{\rm an}\to 0$. Cela munit 
  $\Pi_1^{\rm an}$ d'une structure de $L[\varepsilon]$-module, compatible
 avec la structure de $L[\varepsilon]$-module sur $E_{\rm rig}$. 
 
  La preuve du théorème \ref{EHT} montre
 l'existence de $a,b\in L$ tels que $C(z)=(a+b\varepsilon)z$ pour tout
 $z\in E_{\rm rig}$, o\`{u} $C$ est l'\'{e}l\'{e}ment de Casimir. 
 Puisque $\Pi_1^{\rm an}$ est un quotient de
 $E_{\rm rig}\boxtimes_{\delta}\p1$ (th. \ref{analitic}), on conclut que 
  $C$ agit par $a+b\varepsilon$ sur $\Pi_1^{\rm an}$.
   La proposition \ref{pilc} montre que $C(v)=\frac{k^2-1}{2}v$ pour
 tout $v\in \Pi_1^{\rm alg}$.
 En prenant $v\in \Pi^{\rm alg}\subset \Pi_1^{\rm alg}$ non nul, on obtient
 $av=\frac{k^2-1}{2}v$ et donc $a=\frac{k^2-1}{2}$. Ainsi, on a $b\cdot \varepsilon v=0$ pour
 tout $v\in \Pi_1^{\rm alg}$. On en d\'{e}duit que si $\Pi_1^{\rm alg}\ne \Pi^{\rm alg}$, alors
 $b=0$ et $C$ agit par $\frac{k^2-1}{2}$ sur $E_{\rm rig}\boxtimes_{\delta}\p1$. Or, on dispose 
 \cite[\S VI.3]{CD} d'un accouplement $G$-équivariant parfait entre $E_{\rm rig}\boxtimes_{\delta}\p1$ et 
 $\check{E}_{\rm rig}\boxtimes_{\delta^{-1}}\p1$, ce qui fait que $C$ agit
   aussi par $\frac{k^2-1}{2}$ sur $\check{E}_{\rm rig}\boxtimes_{\delta^{-1}}\p1$. Le théorème
 \ref{EHT} montre que $\check{E}$ est de Hodge-Tate, donc $E$ l'est aussi.

\end{proof}

\begin{remark}
 Le fait que $C$ agit sur $\Pi_1^{\rm an}$ par un élément de $L[\varepsilon]$ 
découle aussi de la pleine fidélité du foncteur $\Pi\to \Pi^{\rm an}$ dans 
${\rm Rep}_L(\delta)$, qui est une conséquence du résultat principal de \cite{CD}. 
On a donc ${\rm End}(\Pi_1^{\rm an})={\rm End}(\Pi_1)=L[\varepsilon]$, car 
$\Pi$ étant absolument irréductible on a ${\rm End}(\Pi)=L$ par le "lemme de Schur $p$-adique"
\cite{DBenjamin}.
\end{remark}

\section{Vecteurs $P$-finis}\label{KirCol}

   \quad Dans ce chapitre on étend et on raffine la théorie du modèle de Kirillov de $\Pi^{\rm alg}$, telle qu'elle est développée dans \cite{Cbigone}.
   On fixe un caractère unitaire $\delta:\qpet\to O_L^*$.
   
  \subsection{Le sous-module $\tilde{D}$ de $D\boxtimes_{\delta}\p1$}\label{Tnorm}
  
 \quad  Pour tout $D\in\Phi\Gamma^{\rm et}(\mathcal{E})$ on note 
  $\tilde{D}=(\tilde{\mathbf{B}}\otimes_{\qp} V(D))^H$ et
    $\tilde{D}^+=(\tilde{\mathbf{B}}^+\otimes_{\qp} V(D))^H$, que l'on munit d'une action du groupe mirabolique $P=\left(\begin{smallmatrix}  \qpet & \qp \\0 & 1\end{smallmatrix}\right)$
     en posant, pour $a\in\zpet, k\in\mathbf{Z}, b\in \qp$ et $z\in \tilde{D}$ (ou $\tilde{D}^+$)
    $$\left(\begin{smallmatrix}  p^ka & b \\0 & 1\end{smallmatrix}\right)z=[(1+T)^b]\varphi^k\circ \sigma_{a}(z).$$
    On étend cette action en une action du groupe de Borel $B$ en demandant que le centre agisse par $\delta$. 
    
  Soit $I\subset \qp$ un système de représentants de $\qp/\zp$ et soit $I_n=I\cap p^{-n}\zp$. D'après \cite[lemme IV.1.2]{Cmirab}, tout \'{e}l\'{e}ment $z$ de $\tilde{D}$ s'\'{e}crit de mani\`{e}re unique sous la forme
  $z=\sum_{i\in I} [(1+T)^i]z_i$, avec $z_i\in D$ et $\lim_{i\to\infty} z_i=0$. On définit une {\it trace de Tate normalisée} $$T_n: \tilde{D}\to \tilde{D}, \quad T_n(z)=\sum_{i\in I_n} [(1+T)^i]z_i.$$
  
   D'après\footnote{Dans loc.cit $\left(\begin{smallmatrix} p^{-n} & 0 \\0 & 1\end{smallmatrix}\right)\cdot\varphi^n(T_n(z))$ est not\'{e}
 ${\rm Res}_{p^{-n}\zp}(z)$.} \cite[lemme II.1.16]{Cbigone}, pour tout $z\in \tilde{D}$, la suite de terme g\'{e}n\'{e}ral $$\left(\begin{smallmatrix} p^{-n} & 0 \\0 & 1\end{smallmatrix}\right)\cdot\varphi^n(T_n(z))=\sum_{i\in I_n} \left(\begin{smallmatrix} 1 & i \\0 & 1\end{smallmatrix}\right)\cdot z_i$$
 converge dans $D\boxtimes_{\delta}\p1$
 et l'application $\iota: \tilde{D}\to D\boxtimes_{\delta}\p1$ ainsi définie est injective et $B$-équivariante. 
 
   \begin{proposition}\label{Bisom} Si $D\in\mathcal{C}_L(\delta)$ alors 
$\iota:\tilde{D}\to D\boxtimes_{\delta}\p1$ envoie 
 $\tilde{D}^+$ dans $\Pi_{\delta^{-1}}(\check{D})^*$, induisant un morphisme
 $\tilde{D}/\tilde{D}^+\to \Pi_{\delta}(D)$, qui est un isomorphisme de $B$-modules topologiques si 
 $D$ n'a pas de facteur de Jordan-Hölder de dimension $1$. 
  
  \end{proposition}
  
  \begin{proof}
   Le fait que $\iota$ envoie $\tilde{D}^+$ dans $\Pi_{\delta^{-1}}(\check{D})^*$ découle de \cite[lemme IV.2.2]{Cmirab} et de la description 
   de $\Pi_{\delta^{-1}}(\check{D})^*$ en fonction de $D$ fournie par \cite[cor. III.22]{CD}. 
   Soit $\check{D}^{\rm nr}=\cap_{n\geq 1}\varphi^n(\check{D})$. Le corollaire III.25 et la proposition II.6 a) de \cite{CD} 
   montrent que $\tilde{D}/\tilde{D}^+\to \Pi_{\delta}(D)$ est injectif, de conoyau isomorphe au dual de $\check{D}^{\rm nr}$.
   Pour conclure, il suffit donc de montrer que $\check{D}^{\rm nr}=0$ quand $D$ n'a pas de facteur de Jordan-Hölder de dimension $1$, qui découle de \cite[cor. II.5.21]{Cmirab}.   
   
  \end{proof}

  \subsection{Espaces de fonctions} 
  
   \quad Soit $X$ un $L_{\infty}[[t]]$-module muni d'une action de $\Gamma$, semi-linéaire par rapport à l'action 
   naturelle de $\Gamma$ sur $L_{\infty}[[t]]$.
    On note ${\rm LP}(\qpet, X)^{\Gamma}$ l'espace des fonctions $\phi: \qpet\to X$ à support compact
  dans $\qp$ et satisfaisant $\phi(ax)=\sigma_{a}(\phi(x))$ pour tous $x\in\qpet$ et $a\in\zpet$. On note ${\rm LP}_c(\qpet, X)^{\Gamma}$
  le sous-espace de ${\rm LP}(\qpet, X)^{\Gamma}$ formé des fonctions nulles au voisinage de $0$. 
  
     Puisque $\sigma_a([(1+T)^b])=[(1+T)^{ab}]$ pour tous
  $a\in \zpet$ et $b\in\qp$, les espaces ${\rm LP}(\qpet, X)^{\Gamma}$ et ${\rm LP}_c(\qpet, X)^{\Gamma}$
sont munis d'une action du groupe de Borel $B$, définie par 
      $$\left(\left(\begin{smallmatrix} a & b \\0 & d\end{smallmatrix}\right)\phi \right) (x)= \delta(d)[(1+T)^{\frac{bx}{d}}]\phi\left(\frac{ax}{d}\right),$$
  
    \begin{remark}\label{declpc}
     L'application $\phi\to (\phi(p^i))_{i\in\mathbf{Z}}$ induit un isomorphisme $L$-linéaire 
     $${\rm LP}_c(\qpet, X)^{\Gamma}\simeq \bigoplus_{i\in\mathbf{Z}} X.$$
    \end{remark}

  \subsection{Vecteurs $U$-finis} \label{Ufini}

\quad  On suppose pour le reste de ce chapitre que $D\in\mathcal{C}_L(\delta)$ et que 
tous les facteurs de Jordan-Hölder de $D$ sont de dimension au moins $2$. On note $\Pi=\Pi_{\delta}(D)$, qui est isomorphe comme 
$B$-module topologique à $\tilde{D}/\tilde{D}^+$ (prop. \ref{Bisom}). On pose $U=\left(\begin{smallmatrix} 1 & \qp \\0 & 1\end{smallmatrix}\right)$ et on rappelle que 
$P=\left(\begin{smallmatrix} \qpet & \qp \\0 & 1\end{smallmatrix}\right)$.

\begin{definition}\label{UPalg}
\label{UPalg}
Pour tout $k\in\mathbf{N}^*$, on pose
$$\Pi^{U, k}=\bigcup_{n\geq 1} {\rm Ker} \left( \left(\begin{smallmatrix} 1 & p^n \\0 & 1\end{smallmatrix}\right) -1  \right)^k$$
et on pose $\Pi^{U-\rm fini}=\bigcup_{k\geq 1} \Pi^{U,k}$.

\end{definition}

\begin{lemma}\label{degre}
 $\Pi^{U,k}$ est l'ensemble des vecteurs $v\in\Pi$ tels que 
 l'application
$x\to \left(\begin{smallmatrix} 1 & x \\0 & 1\end{smallmatrix}\right) v$ soit localement polynomiale de degr\'{e} plus petit que $k$.
 
 \end{lemma}
 
 \begin{proof}
 
 Si $a\in\zp$ et $n\geq 1$, soit $f_{n,a}(x)=\left(\begin{smallmatrix} 1 & a+p^nx \\0 & 1\end{smallmatrix}\right)v$.
  Alors $f_{n,a}:\zp\to \Pi$ est une application continue dont le $k$-i\`{e}me coefficient de Mahler est
   $$a_k(f_{n,a})=\left(\begin{smallmatrix} 1 & a \\0 & 1\end{smallmatrix}\right)\cdot \left(   \left(\begin{smallmatrix} 1 & p^n \\0 & 1\end{smallmatrix}\right)-1\right)^k v.$$
 Le résultat découle alors du théorème de Mahler. 
 \end{proof}
 
 \begin{lemma}\label{piuk}
  Soient $k\in\mathbf{N}^*$ et $v\in\Pi$. Alors $v\in \Pi^{U,k}$ si et seulement si tout rel\`{e}vement $z\in \tilde{D}$ de $v$ appartient \`{a} $\frac{1}{\varphi^n(T)^k}\tilde{D}^+$ pour un
 certain $n\geq 1$. De plus, $\Pi^{U,k}$ est stable sous l'action du Borel $B$. 
 
 \end{lemma}

\begin{proof}
 C'est une simple traduction
  de l'action de $\left(\begin{smallmatrix} 1& \zp \\0 & 1\end{smallmatrix}\right)$ sur $\tilde{D}$ et de l'isomorphisme 
  $B$-équivariant $\Pi\simeq \tilde{D}/\tilde{D}^+$. La seconde assertion découle de la première et du fait que 
  $\cup_{n\geq 1}\frac{1}{\varphi^n(T)^k}\tilde{D}^+$ est un $(\tilde{\mathbf{B}}^+)^H$-module stable sous l'action de $\Gamma$
  et $\varphi^{\pm 1}$.

\end{proof}

    Soit $v\in \Pi^{U-\rm fini}$ et soit $\tilde{z}$ un relèvement de $v$ à $\tilde{D}$. 
    En posant $$\tilde{M}=\bigcup_{n,k\geq 1} \frac{1}{\varphi^n(T)^k}\tilde{D}^+,$$
    on a $\left(\begin{smallmatrix} x & 0 \\0 & 1\end{smallmatrix}\right) \tilde{z}\in \tilde{M}$ pour 
    tout $x\in\qpet$. De plus, l'inclusion  $\tilde{\mathbf{B}}^+\subset\mathbf{B}_{\rm dR}^+$ fournit une injection 
  de $\tilde{D}^+$ dans $\tilde{D}_{\rm dif}^+$ et de $\tilde{M}$ dans $\tilde{D}_{\rm dif}$.
    On définit une fonction
   $$\phi_v: \qpet\to \tilde{D}_{\rm dif}/\tilde{D}^+_{\rm dif},\quad
   \phi_v(x)=\left(\begin{smallmatrix} x & 0 \\0 & 1\end{smallmatrix}\right) \tilde{z}\pmod{ \tilde{D}_{\rm dif}^+}.$$
 Cette fonction est bien définie, car tout autre relèvement de $v$ à $\tilde{D}$ diffère de $\tilde{z}$ par un élément 
 de $\tilde{D}^+\subset \tilde{D}_{\rm dif}^+$. 

  \begin{proposition}
  \label{Kir}
   
    L'espace $\Pi^{U-\rm fini}$ est stable sous l'action de $B$ et l'application $v\to \phi_v$ est une injection $B$-équivariante $\Pi^{U-\rm fini}\to {\rm LP}(\qpet,  \tilde{D}_{\rm dif}/\tilde{D}^+_{\rm dif} )^{\Gamma}$.
    
  \end{proposition}

\begin{proof} La stabilité sous l'action de $B$ découle du lemme \ref{piuk}. 
Soit $\tilde{z}\in \frac{1}{\varphi^n(T)^k}\tilde{D}^+$ un relèvement de $v\in \Pi^{U-\rm fini}$.
   Le a) du lemme \ref{frobdr} montre que
$\phi_v(p^{-m})=0$ pour tout $m>n$, et donc $\phi_v$ est \`{a} support compact dans $\qp$. Il est clair que $\phi_v(ax)=\sigma_a(\phi_v(x))$. Si
$\phi_v=0$, alors $\varphi^m(\tilde{z})\in\tilde{D}_{\rm dif}^+$ pour tout $m\in\mathbf{Z}$. Ainsi, 
$\tilde{x}=\varphi^n(T)^k\cdot\tilde{z}\in \tilde{D}^+$ satisfait $\varphi^m(\varphi^{-n}(\tilde{x}))\in \varphi^m(T)^k\tilde{D}_{\rm dif}^+
\subset t^k\mathbf{B}_{\rm dR}^+\otimes_{\qp} V(D)$ pour tout $m\geq 0$, et le b) du lemme \ref{frobdr} permet de conclure que
$\varphi^{-n}(\tilde{x})\in T^k\tilde{D}^+$, donc $\tilde{z}\in\tilde{D}^+$ et $v=0$.  La $B$-équivariance est une simple traduction de l'action de $B$ sur $\tilde{D}$, ce qui permet de conclure.

\end{proof}

\subsection{Vecteurs $P$-finis}\label{Pfini}
  
\begin{definition} On note $T_n=\sum_{i=0}^{p^n-1} \left(\begin{smallmatrix} 1 & i \\0 & 1\end{smallmatrix}\right) \in {\rm End}(\Pi)$ et on note 
 $$\Pi_c^{U-\rm fini}=\bigcup_{n,k\geq 1}\bigcup_{ m\in\mathbf{Z}} {\rm Ker} \left( T_n^k\circ \left(\begin{smallmatrix} p^m & 0 \\0 & 1\end{smallmatrix}\right) \right).$$

\end{definition}

\begin{proposition}\label{ufinicompact} On a 
$$\Pi_c^{U-\rm fini}=\{v\in \Pi^{U-\rm fini}| \quad \phi_v\in {\rm LP}_c(\qpet, \tilde{D}_{\rm dif}/\tilde{D}_{\rm dif}^+)^{\Gamma}\}$$
et $\Pi_c^{U-\rm fini}$ est stable sous l'action de $B$. 
\end{proposition}

\begin{proof} Supposons que $v\in \Pi_c^{U-\rm fini}$ et soit $\tilde{z}\in\tilde{D}$ un rel\`{e}vement de $v$. La condition $T_n^k\circ \left(\begin{smallmatrix} p^m & 0 \\0 & 1\end{smallmatrix}\right)v=0$ est équivalente à $\left(\frac{\varphi^n(T)}{T}\right)^k\varphi^m(\tilde{z})\in \tilde{D}^+$, ou encore à
 $\tilde{z}\in \left(\frac{\varphi^{-m}(T)}{\varphi^{n-m}(T)}\right)^k\tilde{D}^+$. Ceci entraîne d'une part que $v\in \Pi^{U-\rm fini}$ (lemme \ref{piuk})
 et d'autre part que $\varphi^j(\tilde{z})\in \tilde{D}_{\rm dif}^+$ pour tout $j\geq m$ (lemme 
\ref{frobdr}). Puisque $\phi_v(ax)=\sigma_{a}(\phi_v(x))$, on conclut que $\phi_v$ est nulle sur $p^m \zp-\{0\}$, donc 
$\phi_v\in  {\rm LP}_c(\qpet, \tilde{D}_{\rm dif}/\tilde{D}_{\rm dif}^+)^{\Gamma}$.

  Pour la réciproque, supposons que $v\in \Pi^{U,k}$ et qu'il existe $m\geq 1$ tel que $\phi_v$ s'annule sur $p^m \zp-\{0\}$. Soit $\tilde{z}=\frac{\tilde{x}}{\varphi^n(T)^k}$
   un rel\`{e}vement de $v$, avec $\tilde{x}\in \tilde{D}^+$ (lemme \ref{piuk}). Puisque $\phi_v(p^{j+m})=0$ pour tout $j\geq 0$, on a  $\varphi^{m+j}(\tilde{z})\in \tilde{D}_{\rm dif}^+$, donc $\varphi^j(\varphi^m(\tilde{x}))\in t^k\tilde{D}_{\rm dif}^+$ pour tout
   $j\geq 0$. On d\'{e}duit du lemme \ref{frobdr} que $\varphi^m(\tilde{x})\in T^k\tilde{D}^+$ et donc
   $\varphi^m(\tilde{z})\in \left(\frac{T}{\varphi^{m+n}(T)}\right)^k \tilde{D}^+$. Mais cela est \'{e}quivalent \`{a}
   $T_{m+n}^k \circ  \left(\begin{smallmatrix} p^m & 0 \\0 & 1\end{smallmatrix}\right)v=0$, ce qui montre que 
   $v\in \Pi_c^{U-\rm fini}$ et finit la preuve de la première assertion. La seconde s'en déduit, en utilisant la proposition \ref{Kir}.

\end{proof}

 \begin{definition} On note $\Pi_c^{P-\rm fini}$
 le sous-espace de $\Pi_c^{U-\rm fini}$ formé des $v$ tels que $L\left [ \left(\begin{smallmatrix} \zpet & 0 \\0 & 1\end{smallmatrix}\right)\right] v$ soit de dimension finie sur $L$.
\end{definition}

 Nous allons d\'{e}crire l'espace $\Pi_c^{P-\rm fini}$ de manière explicite, en utilisant la th\'{e}orie de Hodge $p$-adique.
 Soit $D_{\rm dif,\infty}/ D_{\rm dif,\infty}^+$ la limite inductive (c'est aussi la r\'{e}union croissante)
  des $D_{{\rm dif},n}/D_{{\rm dif},n}^+$. C'est naturellement un sous $L_{\infty}[[t]]$-module de $\tilde{D}_{\rm dif}/\tilde{D}_{\rm dif}^+$.

   \begin{proposition}\label{picpfct}
    L'application $v\to \phi_v$ induit un isomorphisme de $B$-modules 
    $$\Pi^{P-\rm fini}_c\simeq {\rm LP}_c\left(\qpet, D_{{\rm dif}, \infty}/D_{{\rm dif}, \infty}^+\right)^{\Gamma}.$$
   
   \end{proposition}

  \begin{proof}
  
    Si $M$ est un $L[\Gamma]$-module, on note $M^{\Gamma-\rm fini}$ l'ensemble des $x\in M$ tels que 
    $L[\Gamma]x$ soit de dimension finie sur $L$. Si $v\in \Pi_c^{P-\rm fini}$, alors 
     $L\left [ \left(\begin{smallmatrix} \zpet & 0 \\0 & 1\end{smallmatrix}\right)\right] v$ est de dimension finie sur $L$, donc il en de même de 
     $L\left [ \left(\begin{smallmatrix} \zpet & 0 \\0 & 1\end{smallmatrix}\right)\right] \phi_v$. Comme $\phi$ est $\Gamma$-invariante, on obtient $\phi_v(x)\in \left(\tilde{D}_{\rm dif}/ \tilde{D}_{\rm dif}^+\right)^{\Gamma-\rm fini}$
     pour tout $x\in \qpet$. Le lemme \ref{senfini} ci-dessous permet de conclure que 
     $v\to \phi_v$ induit une injection $B$-équivariante de $\Pi_c^{P-\rm fini}$ dans 
    ${\rm LP}_c(\qpet, D_{\rm dif,\infty}/D_{\rm dif,\infty}^+)^{\Gamma}$.
    
\begin{lemma}\label{senfini}
 Pour tout $D\in\Phi\Gamma^{\rm et}(\mathcal{E})$ on a
  $$\left(\tilde{D}_{\rm dif}/ \tilde{D}_{\rm dif}^+\right)^{\Gamma-\rm fini}=D_{\rm dif,\infty}/D_{\rm dif,\infty}^+.$$
\end{lemma}

\begin{proof}
 Soit $V=V(D)$ et notons, pour simplifier, $X=\tilde{D}_{\rm dif}^+$ et $Y=D_{\rm dif,\infty}^+$.
 La suite exacte $$0\to t\mathbf{B}_{\rm dR}^+\otimes_{\qp} V\to \mathbf{B}_{\rm dR}^+\otimes_{\qp} V\to \mathbf{C}_p\otimes_{\qp} V\to 0$$
 et la nullité de $H^1(H, t\mathbf{B}_{\rm dR}^+\otimes_{\qp} V)$ (qui découle de \cite[th. IV.3.1]{Cannals}) montrent que 
 $X/tX\simeq (\mathbf{C}_p\otimes_{\qp} V)$. Par définition, on a 
$Y/tY=D_{\rm Sen}=\cup_{n\geq 1} D_{{\rm Sen},n}$, donc d'apr\`{e}s Sen \cite{Sen} on a
 $(X/tX)^{\Gamma-\rm fini}=Y/tY$. Comme $t$ est vecteur propre pour $\Gamma$, on a aussi
 $(t^{-1}X/X)^{\Gamma-\rm fini}=t^{-1}Y/Y$. En utilisant les suites exactes
 (dont les applications-induites par la multiplication par $t^{k-1}$-sont $\Gamma$-\'{e}quivariantes, \`{a} des \'{e}l\'{e}ments de $\zpet$ pr\`{e}s)
 $$0\to t^{-k+1}X/X\to t^{-k}X/X\to t^{-1}X/X\to 0,$$
$$ 0\to t^{-k+1}Y/Y\to t^{-k}Y/Y\to t^{-1}Y/Y\to 0,$$
 on obtient par r\'{e}currence sur $k\geq 1$ que $(t^{-k}X/X)^{\Gamma-\rm fini}=t^{-k}Y/Y$. Le r\'{e}sultat se d\'{e}duit en passant \`{a} la limite
inductive sur $k$.

\end{proof}

Il nous reste à prouver la surjectivité de $\Pi^{P-\rm fini}_c\to {\rm LP}_c(\qpet, D_{{\rm dif}, \infty}/D_{{\rm dif}, \infty}^+)^{\Gamma}$. La preuve est très semblable à 
celle de \cite[prop. 12]{Annalen}. En utilisant la remarque \ref{declpc} et en raisonnant comme dans 
 le premier paragraphe de la preuve 
de loc.cit., on se ramène à montrer que pour tout $x\in D_{{\rm dif}, \infty}/D_{{\rm dif}, \infty}^+$ on peut trouver 
$v\in \Pi^{P-\rm fini}_c$ tel que $\phi_v(p^i)=1_{i=0} x$ pour tout $i\in\mathbf{Z}$. 

   Soient $k\geq 1$ et $n\geq m(D)$ tels que $x\in t^{-k} D_{{\rm dif}, n}^+/D_{{\rm dif},n}^+$, et soit $\hat{x}\in t^{-k} D_{{\rm dif}, n}^+$ un relèvement de $x$. Soit $\omega=\frac{T}{\varphi^{-1}(T)}\in \tilde{\mathbf{A}}^+$. Puisque $\omega^k \hat{x}\in \tilde{D}_{\rm dif}^+$, le lemme $6$ de \cite{Annalen} fournit $y\in \tilde{D}^+$ tel que $y-\omega^k \hat{x}\in \omega^k \tilde{D}_{\rm dif}^+$.
Soit $v$ l'image de $\omega^{-k} y\in \tilde{D}$ dans $\Pi\simeq \tilde{D}/\tilde{D}^+$. Le lemme \ref{frobdr} et le choix de $y$ montrent que $\phi_v(p^i)=1_{i=0} x$ pour tout $i$. 
Comme $T^k\omega^{-k}\in \mathbf{\tilde{A}}^+$, $(\left(\begin{smallmatrix} 1 & 0 \\0 & 1\end{smallmatrix}\right)-1)^k$ tue $v$, donc $v\in \Pi^{U-\rm fini}$. Comme 
$x$ est $\Gamma$-fini (lemme \ref{senfini}) et $\phi_v(p^i)=1_{i=0} x$, on obtient que $v$ est $\Gamma$-fini, donc 
$v\in \Pi^{P-\rm fini}$, ce qui permet de conclure.
\end{proof}

    \begin{corollary}\label{picp}
  $ \Pi^{P-\rm fini}_c$ est stable sous l'action de $B$ et l'application $v\to (\phi_v(p^i))_{i\in\mathbf{Z}}$ induit une bijection
    $$\Pi^{P-\rm fini}_c\to \bigoplus_{i\in\mathbf{Z}} D_{\rm dif,\infty}/D_{\rm dif,\infty}^+.$$
  \end{corollary}
  
  \begin{proof}
   C'est une conséquence immédiate de la proposition précédente et du fait que $D_{\rm dif,\infty}/D_{\rm dif,\infty}^+$ est un sous $L_{\infty}[[t]]$-module de
   $\tilde{D}_{\rm dif}/\tilde{D}_{\rm dif}^+$, stable sous l'action de $\Gamma$.
     \end{proof}
  
  \subsection{L'espace $\Pi_c^{P-\rm alg}$} 
  
    \quad On suppose dans ce $\S$ que les poids de Hodge-Tate de $D$ sont $\leq 1$
    et on fixe $k\in\mathbf{N}^*$. 
        
   \begin{definition}
  a) On note $\Pi_{c,k}^{P-\rm alg}$ le sous-espace de $\Pi_c^{P-\rm fini}$ formé des vecteurs $v$ tels que 
$\mu_{n,k}(v)=0$ pour tout $n$ assez grand, où 
$$\mu_{n,k}=\prod_{i=0}^{k-1} \left(\left(\begin{smallmatrix} 1+p^n & 0 \\0 & 1\end{smallmatrix}\right) -(1+p^n)^{i+1-k}\right).$$

b) On note $(D_{{\rm dif},\infty}/D_{{\rm dif},\infty}^+)^{k,\rm alg}$ la réunion des 
sous-modules $(D_{{\rm dif},n}/D_{{\rm dif},n}^+)^{\mu_{n,k}=0}$ de $D_{{\rm dif},\infty}/D_{{\rm dif},\infty}^+$, où $$\mu_{n,k}=\prod_{i=0}^{k-1}(\sigma_{1+p^n}-(1+p^n)^{i+1-k}).$$
   \end{definition}
   
   \begin{remark} \label{ddifalg}
   
   a) On se permet de noter deux objets différents par $\mu_{n,k}$, puisque l'action de 
   $\left(\begin{smallmatrix} 1+p^n & 0 \\0 & 1\end{smallmatrix}\right)$ est induite par celle de $\sigma_{1+p^n}$.
   Comme ces opérateurs agissent sur des espaces différents, cela ne crée aucun risque de confusion.
   
  b) On vérifie sans mal que $\mu_{n,k}$ divise $\mu_{n+1,k}$
    dans $O_L[[\left(\begin{smallmatrix} \zpet & 0 \\0 & 1\end{smallmatrix}\right)]]$. 
    On en déduit que la suite des modules $(\Pi_c^{P-\rm fini})^{\mu_{n,k}=0}$ est croissante. 
    Le même argument montre que la suite de modules 
    $(D_{{\rm dif},n}/D_{{\rm dif},n}^+)^{\mu_{n,k}=0}$ est croissante.
    
  c) Le lemme des noyaux fournit une décomposition 
     $$\left(D_{{\rm dif},n}/D_{{\rm dif},n}^+\right)^{\mu_{n,k}=0}=\bigoplus_{i=0}^{k-1}
 \left(D_{{\rm dif},n}/ D_{{\rm dif},n}^+\right)^{\sigma_{1+p^n}=(1+p^n)^{i+1-k}}.$$
   \end{remark}
   
   \begin{lemma}\label{tktue}
   a) Si $0\leq i<k$, alors $t^{k-i}$ tue $\left(D_{{\rm dif},n}/D_{{\rm dif},n}^+\right)^{\sigma_{1+p^n}=(1+p^n)^{i+1-k}}$ pour $n\geq m(D)$.
   
   b) $t^k$ tue $\left(D_{{\rm dif},n}/D_{{\rm dif},n}^+\right)^{\mu_{n,k}=0}$ pour $n\geq m(D)$
   \end{lemma}
   
   \begin{proof} Le point b) découle de a) et du point c) de la 
     remarque \ref{ddifalg}. Passons au a) et considérons $z\in D_{{\rm dif},n}$ tel que
    $(\sigma_{1+p^n}-(1+p^n)^{i+1-k})z\in D_{{\rm dif},n}^+$. Il faut montrer que 
    $t^{k-i}z\in D_{{\rm dif},n}^+$.  
    Si ce n'est pas le cas, alors on peut écrire
    $z=\frac{x}{t^N}$, avec $x\in D_{{\rm dif},n}^+-tD_{{\rm dif},n}^+$ et $N>k-i$. Comme $(\sigma_{1+p^n}-(1+p^n)^{i+1-k})z\in D_{{\rm dif},n}^+$,
   on obtient en particulier $\sigma_{1+p^n}(x)-(1+p^n)^{i+1-k+N}x\in tD_{{\rm dif},n}^+$, donc 
  l'image $\overline{x}$ de $x$ dans $D_{{\rm Sen},n}$ est un vecteur propre de $\Theta_{{\rm Sen},D}$, de valeur
  propre $i+1-k+N\geq 2$. Cela contredit notre hypothèse que les poids de Hodge-Tate de $D$ sont $\leq 1$.
   \end{proof}

   \begin{proposition} \label{pipalg} L'espace $\Pi_{c,k}^{P-\rm alg}$ est stable sous l'action de $B$ et 
  l'application $v\to\phi_v$ induit un  
    isomorphisme $B$-équivariant 
     $$\Pi_{c,k}^{P-\rm alg}\simeq {\rm LP}_c(\qpet, (D_{{\rm dif},\infty}/D_{{\rm dif},\infty}^+)^{k,\rm alg}).$$
   
   \end{proposition}
   
   \begin{proof}
    On déduit formellement de la proposition \ref{picpfct} et du b) de la remarque \ref{ddifalg} que $v\to\phi_v$ induit un isomorphisme $L$-linéaire 
    $$\Pi_{c,k}^{P-\rm alg}\simeq {\rm LP}_c(\qpet, (D_{{\rm dif},\infty}/D_{{\rm dif},\infty}^+)^{k,\rm alg}).$$
    Pour démontrer que $\Pi_{c,k}^{P-\rm alg}$ est stable par $B$ et conclure la preuve de la proposition, il suffit de montrer
    que $(D_{{\rm dif},\infty}/D_{{\rm dif},\infty}^+)^{k,\rm alg}$ est un sous $L_{\infty}[[t]]$-module de $D_{{\rm dif},\infty}/D_{{\rm dif},\infty}^+$, stable par 
    $\Gamma$. La stabilité par $\Gamma$ est claire, puisque $\Gamma$ commute à $\mu_{n,k}$. Pour montrer que 
    $(D_{{\rm dif},\infty}/D_{{\rm dif},\infty}^+)^{k,\rm alg}$ est un $L_{\infty}[[t]]$-module, nous avons besoin de quelques préliminaires.

    \begin{lemma}\label{technical}
    Si $0\leq i<k$
    et $z\in D_{{\rm dif},n}$ satisfont $(\sigma_{1+p^n}-(1+p^n)^{i+1-k})z\in D_{{\rm dif},n}^+$, alors 
    $\mu_{n,k}(fz)\in D_{{\rm dif},n}^+$ pour tout $f\in L_n[[t]]$.
    \end{lemma}
    
    \begin{proof}
      Notons pour simplifier $a=1+p^n$. La congruence $\sigma_a(z)\equiv a^{i+1-k}z\pmod {D_{{\rm dif},n}^+}$ entraîne
    $$\mu_{n,k}(fz)\equiv \prod_{j=0}^{k-1}(a^{i+1-k}\sigma_a-a^{j+1-k})f\cdot z\pmod {D_{{\rm dif},n}^+}.$$
    Ensuite, on vérifie sans mal que 
    $$\prod_{j=0}^{k-1}(a^{i+1-k}\sigma_a-a^{j+1-k})f\in t^{k-i} L_n[[t]],$$
    et le résultat suit du lemme \ref{tktue}.

    \end{proof}
    
     Revenons à la preuve du fait que $(D_{{\rm dif},\infty}/D_{{\rm dif},\infty}^+)^{k,\rm alg}$ est un sous $L_{\infty}[[t]]$-module de $D_{{\rm dif},\infty}/D_{{\rm dif},\infty}^+$.
     Le lemme \ref{technical} montre que 
    $(D_{{\rm dif},n}/D_{{\rm dif},n}^+)^{\mu_{n,k}=0}$ est un $L_n[[t]]$-module. 
    Le résultat s'ensuit, en passant à la limite et en utilisant le lemme 
    \ref{tktue} et le fait que $L_{\infty}[[t]]/t^k$ est la réunion des 
    $L_n[[t]]/t^k$.

   \end{proof}
   
   \begin{corollary}\label{blah}
    L'application $v\to (\phi_v(p^i))_{i\in\mathbf{Z}}$ induit un isomorphisme de $L$-espaces vectoriels
    $$\Pi_{c,k}^{P-\rm alg}\simeq \bigoplus_{i\in\mathbf{Z}} (D_{{\rm dif},\infty}/D_{{\rm dif},\infty}^+)^{k,\rm alg}.$$
   \end{corollary}
   
   \begin{lemma}\label{uplus}
    Pour tout $v\in\Pi_{c,k}^{P-\rm alg}$ on a $\left(\left(\begin{smallmatrix} 1 & p^n \\0 & 1\end{smallmatrix}\right)-1\right)^kv=0$ pour tout 
    $n$ assez grand.
   
   \end{lemma}
   
   \begin{proof} Soient $i\leq j\in\mathbf{Z}$ tels que $\phi_v(p^m)=0$ pour $m\notin [i,j]$. 
  Si $n\geq -i$ et $m\in [i,j]$, alors $([(1+T)^{p^{n+m}}]-1)^k=\varphi^{m+n}(T)^k\in t^k L_n[[t]]$. Puisque 
  $\phi_v(p^m)$ est tué par
   $t^k$ par le lemme \ref{tktue} et la proposition \ref{pipalg}, on obtient $([(1+T)^{p^{n+m}}]-1)^k\phi_v(p^m)=0$
   pour tout $m\in\mathbf{Z}$ et $n\geq -i$, d'où le résultat.

      \end{proof}
    
    \section{Le théorème de dualité de Colmez}
 
  \quad Dans ce chapitre on étend (avec une preuve différente) un résultat de dualité de Colmez. Ce résultat fait le lien entre les deux chapitres précédents et il est au coeur de la preuve du théorème \ref{main}. Il permet aussi de faire 
   le lien entre les facteurs epsilon des représentations supercuspidales de $G$ et l'exponentielle de Perrin-Riou \cite{DIwasawa}.  
   On utilise enfin ce théorème de Colmez pour démontrer un résultat de densité des vecteurs analytiques $P$-finis. 
    
       \subsection{Traces de Tate normalisées}\label{Tatenorm}
   
   \quad Nous avons défini dans le § \ref{Tnorm} des traces de Tate normalisées $T_n: \tilde{D}\to \tilde{D}$ pour tout 
    $D\in \Phi\Gamma^{\rm et}(\mathcal{E})$ et tout $n\geq 0$. En général, il n'est pas vrai que 
    $T_n(\tilde{D}^+)\subset \tilde{D}^+$, mais c'est le cas \cite[prop. 8.5]{CvectdR} si $D=\mathcal{E}$ est le $(\varphi,\Gamma)$-module trivial. On obtient ainsi des applications $\qp$-linéaires $T_n: \tilde{\mathbf{B}}^H\to \tilde{\mathbf{B}}^H$ et $T_n: (\tilde{\mathbf{B}}^+)^H\to (\tilde{\mathbf{B}}^+)^H$. 
    Le r\'{e}sultat suivant fait l'objet de \cite[prop. V.4.5]{Cannals}.

\begin{proposition}\label{Tate}
  Soit $n\geq 1$. L'application $T_n: (\tilde{\mathbf{B}}^+)^H\to (\tilde{\mathbf{B}}^+)^H$ se prolonge par continuité en une application $F_n[[t]]$-linéaire
   $T_n:(\mathbf{B}_{\rm dR}^+)^H\to F_n[[t]]\subset (\mathbf{B}_{\rm dR}^+)^H $. 
\end{proposition}

  On d\'{e}finit alors $T_{0, \rm dR}:\mathbf{B}_{\rm dR}^H\to \qp((t))\subset \mathbf{B}_{\rm dR}^H$ par
  $$T_{0, \rm dR}=\frac{1}{p} {\rm Tr}_{F_1((t))/\qp((t))}\circ T_1,$$
  $T_1$ étant l'unique extension $F_1((t))$-linéaire de l'application fournie par la proposition \ref{Tate}. Pour tout $n\geq 1$ on a 
  $$T_{0, \rm dR}|_{F_n((t))}=\frac{1}{p^n} {\rm Tr}_{F_n((t))/\qp((t))}.$$

  \subsection{Comparaison de deux accouplements}\label{dualitefond}
      
   \quad Soit $D\in \Phi\Gamma^{\rm et}(\mathcal{E})$ et soit $V=V(D)$. L'accouplement ${\rm Gal}(\overline{\qp}/\qp)$-\'{e}quivariant parfait $V^*(1)\times V\to L(1)$
 induit un accouplement $\tilde{\mathbf{B}}^H\otimes_{\qp} L$-bilin\'{e}aire, $\Gamma$-\'{e}quivariant $$\langle \,\,,\,\rangle: \tilde{\check{D}}\times \tilde{D}=(\tilde{\mathbf{B}}\otimes_{\qp} V^*(1))^H\times (\tilde{\mathbf{B}}\otimes_{\qp} V)^H\to (\tilde{\mathbf{B}}^H\otimes_{\qp} L)(1).$$
 On note $\frac{dT}{1+T}$ la base canonique du $\Gamma$-module $\qp(1)$ et on fait agir $\varphi$ sur $\frac{dT}{1+T}$ par
 $\varphi(\frac{dT}{1+T})=\frac{dT}{1+T}$. On identifie
 $M(1)$ \`{a} $M\frac{dT}{1+T}$ pour tout $\Gamma$-module $M$. Par définition, l'image de l'application $T_0: \tilde{\mathbf{B}}^H\to \tilde{\mathbf{B}}^H$ (voir le § \ref{Tatenorm}) est contenue dans $\mathcal{E}$,
 d'où une application $\Gamma$-\'{e}quivariante
 $$T_0=(T_0\otimes {\rm id})(1): (\tilde{\mathbf{B}}^H\otimes_{\qp} L)(1)\to \mathcal{E}\frac{dT}{1+T}.$$ 
 On d\'{e}finit
 un accouplement entre $\tilde{\check{D}}$ et $\tilde{D}$ par
 $$\{\,\,,\,\}: \tilde{\check{D}}\times \tilde{D}\to L, \quad \{\check{z}, z\}={\rm res}_0(T_0(\langle \sigma_{-1}(\check{z}), z\rangle)),$$ avec la notation habituelle ${\rm res}_0(f)=a_{-1}$ si
 $f=(\sum_{n\in\mathbf{Z}} a_nT^n)dT$. Il découle de \cite[prop. I.2.2]{Cmirab} que cet accouplement est $\varphi$ et $\Gamma$ invariant
 (sa restriction \`{a} $\check{D}\times D$ est celle utilis\'{e}e dans
 \cite{Cbigone}). 
 
 L'accouplement $V^*(1)\times V\to L(1)$
  induit aussi un accouplement
  $\Gamma$-\'{e}quivariant parfait $$\langle \,\,,\,\rangle: \tilde{\check{D}}_{\rm dif}\times \tilde{D}_{\rm dif}\to
  (L\otimes_{\qp} \mathbf{B}_{\rm dR}^H)(1)=(L\otimes_{\qp} \mathbf{B}_{\rm dR}^H) dt,$$
  $dt$ \'{e}tant\footnote{Comme $t=\log(1+T)$, on a bien $dt=\frac{dT}{1+T}$, compatible avec
  l'action de $\Gamma$.} une base de $\qp(1)$. Cet accouplement envoie $\tilde{\check{D}}^+_{\rm dif}\times
  \tilde{D}^+_{\rm dif}$ dans $(L\otimes_{\qp} (\mathbf{B}_{\rm dR}^+)^H)dt$. On d\'{e}finit alors un accouplement $\{\,\,,\,\}_{\rm dif}: \tilde{\check{D}}_{\rm dif}\times \tilde{D}_{\rm dif}\to L$
   en posant 
    $$\{\check{z},z\}_{\rm dif}={\rm res}_0( T_{0, \rm dR}(\langle \sigma_{-1}(\check{z}), z\rangle)).$$

        \begin{proposition}\label{reciproc}
 Soient $n,k\in\mathbf{N}^*$, $a\in\mathbf{Z}$ et $\check{z}\in \tilde{\check{D}}^+$, $z\in \left(\frac{\varphi^a(T)}{\varphi^n(T)}\right)^k\tilde{D}^+$.
 Alors $$\{\check{z},z\}=\sum_{j\in\mathbf{Z}} \{\varphi^{-j}(\check{z}), \varphi^{-j}(z)\}_{\rm dif},$$ la somme n'ayant qu'un nombre fini
 de termes non nuls.
\end{proposition}

\begin{proof} On peut supposer\footnote{ Si $a\ne 0$, remplacer $z$ et $\check{z}$ par $\varphi^{-a}(z)$ et
 $\varphi^{-a}(\check{z})$, ainsi que $n$ par $\max(1,n-a)$; l'\'{e}galit\'{e} \`{a} d\'{e}montrer
est la m\^{e}me, car $\{\,\,,\,\}$ est $\varphi$-invariant.} que $a=0$. Posons 
$\langle \sigma_{-1}(\check{z}), z\rangle=f\frac{dT}{1+T}$, avec $f\in \left(\frac{T}{\varphi^n(T)}\right)^k (L\otimes_{\qp}
 (\tilde{\mathbf{B}}^+)^H)$. Alors $$\langle \sigma_{-1}(\varphi^{-j}(\check{z})), \varphi^{-j}(z)\rangle=\varphi^{-j}(f)dt\in (L\otimes_{\qp} \mathbf{B}_{\rm dR}^H)dt$$
et $\varphi^{-j}(f)\in (L\otimes_{\qp} \mathbf{B}_{\rm dR}^+)^H$ pour $j\notin [1,n]$ (lemme \ref{frobdr}), ce qui fait que seuls les termes
d'indice $j\in [1,n]$ dans la somme sont non nuls. Dans la suite on fixe un tel $j$.

   Notons $g=T_0(f)\in \left(\frac{T}{\varphi^n(T)}\right)^k \mathcal{E}^+$, de telle sorte que
   $\{\check{z},z\}={\rm res}_0( g\frac{dT}{1+T})$. Puisque $j\geq 1$, on
   a $T_1\circ T_j=T_1$ et $\varphi^j\circ T_j=T_0\circ \varphi^j$, donc 
   $\varphi^{-j}\circ T_0=T_j\circ \varphi^{-j}$.
  On en déduit que 
   $$T_1(\varphi^{-j}(g))=T_1(\varphi^{-j}(T_0(f)))=T_1\circ T_j (\varphi^{-j}(f))=T_1(\varphi^{-j}(f)),$$
  donc $T_{0, \rm dR}(\varphi^{-j}(g))=T_{0, \rm dR}(\varphi^{-j}(f))$ et
    $$\{\varphi^{-j}(\check{z}), \varphi^{-j}(z)\}_{\rm dif}={\rm res}_0(T_{0, \rm dR}(\varphi^{-j}(g))dt).$$

    Si $A_j$ est l'ensemble des racines primitives d'ordre $p^j$ de l'unit\'{e}, alors
$$T_{0, \rm dR}(\varphi^{-j}(g))=\frac{1}{p^j} {\rm Tr}_{L_j((t))/L((t))} (g(\varepsilon^{(j)}e^{t/ p^{j}}-1))=
\frac{1}{p^j}\sum_{\zeta\in A_j} g(\zeta e^{t/p^{j}}-1)$$ et, en faisant la substitution $t\to p^{j-n}t$, on obtient finalement
 $$ \{\varphi^{-j}(\check{z}), \varphi^{-j}(z)\}_{\rm dif}=\frac{1}{p^n}\sum_{\zeta\in A_j} {\rm res}_0(g(\zeta e^{t/ p^{n}}-1)dt),$$ ce qui permet de conclure que
 $$\sum_{j\in\mathbf{Z}} \{\varphi^{-j}(\check{z}), \varphi^{-j}(z)\}_{\rm dif}=\frac{1}{p^n}\sum_{\zeta\in \mu_{p^n}} {\rm res}_0(g(\zeta e^{t/p^{n}}-1)dt)$$
$$= {\rm res}_0\left(\frac{1}{p^n}\sum_{\zeta\in \mu_{p^n}}  g(\zeta e^{t/p^{n}}-1) dt\right)={\rm res}_0(\psi^n(g)(e^t-1)dt)=$$ $${\rm res}_0\left(\psi^n(g)\frac{dT}{1+T}\right)={\rm res}_0\left(g\frac{dT}{1+T}\right)=\{\check{z},z\}.$$
Les derni\`{e}res \'{e}galit\'{e}s utilisent le changement de variable
 $T=e^t-1$ et l'\'{e}galit\'{e} \cite[prop. I.2.2]{Cmirab} ${\rm res}_0(f\frac{dT}{1+T})={\rm res}_0(\psi(f)\frac{dT}{1+T})$ pour tout $f\in\mathcal{R}$.

\end{proof}

\subsection{Un résultat de dualité}\label{dualKirillov}

 \quad   Le théorème \ref{dualKir} ci-dessous étend et raffine \cite[prop. VI.5.12]{Cbigone}, avec une preuve différente de celle de loc.cit. Son énoncé demande pas mal de préliminaires.
 Rappelons (voir \ref{dualitefond}) que l'on dispose d'un
accouplement $\{\,\,,\,\}_{\rm dif}$. 
Il induit un accouplement $\Gamma$-invariant parfait entre $\check{D}_{{\rm dif},n}$ et
   $D_{{\rm dif},n}$, $\check{D}_{{\rm dif},n}^+$ et $D_{{\rm dif},n}^+$ étant exactement orthogonaux pour $n$ assez grand \cite[lemme VI.3.3]{Cbigone}, d'où un accouplement 
   $\{\,\,,\,\}_{\rm dif}$ entre $\check{D}_{{\rm dif},n}^+$ et $D_{{\rm dif},n}/D_{{\rm dif},n}^+$.

On fixe $D\in\mathcal{C}_L(\delta)$
 et on pose $\Pi=\Pi_{\delta}(D)$. On suppose que les facteurs de Jordan-Hölder de $D$ sont tous de dimension au moins $2$, de telle sorte que les résultats du
 § \ref{Pfini} s'appliquent à l'espace $\Pi^{P-{\rm fini}}$. D'après le corollaire VI.13 de \cite{CD}, 
 il existe $m(D)$ assez grand tel que $(\Pi^{\rm an})^*\subset \check{D}^{]0,r_{m(D)}]}\boxtimes_{\delta^{-1}}\p1$. 
  Quitte à augmenter $m(D)$, on peut supposer 
 que le morphisme $\varphi^{-n}= \check{D}^{]0,r_{n}]}\to \tilde{\check{D}}_{\rm dif}^+$ de localisation en 
 $\varepsilon^{(n)}-1$ est défini pour $n\geq m(D)$ (voir le § \ref{con}).  
On peut donc d\'{e}finir pour $\check{z}\in (\Pi^{\rm an})^*$, $n\geq m(D)$ et $j\in\mathbf{Z}$ $$i_{j,n}(\check{z})=\varphi^{-n}\left( {\rm Res}_{\zp}\left( \left(\begin{smallmatrix} p^{n-j} & 0\\0 & 1\end{smallmatrix}\right)\check{z}\right)\right)\in \check{D}_{{\rm dif},n}^+. $$
On note $\{\,\,,\,\}_{\p1}$ l'accouplement naturel entre $\Pi^*$ (dual topologique de $\Pi$) et $\Pi$, ainsi qu'entre 
$(\Pi^{\rm an})^*$ et $\Pi^{\rm an}$. Il est en fait induit par un accouplement 
$G$-équivariant parfait $\{\,\,,\,\}_{\p1}$ entre $\check{D}\boxtimes_{\delta^{-1}}\p1$ et $D\boxtimes_{\delta}\p1$ (resp. 
entre $\check{D}_{\rm rig}\boxtimes_{\delta^{-1}}\p1$ et $D_{\rm rig}\boxtimes_{\delta}\p1$) (voir 
le \S \, III.6 de \cite{CD} pour ces accouplements). 

 \begin{theorem}\label{dualKir}
  Soit $v\in \Pi_c^{P-\rm fini}$ et soit $n\geq m(D)$ tel que $\phi_v(p^{j})\in D_{{\rm dif},n}/ D_{{\rm dif},n}^+$
  pour tout $j\in\mathbf{Z}$. Alors $v\in \Pi^{\rm an}$ et pour tout $\check{z}\in (\Pi^{\rm an})^*$ on a
   $$\{\check{z},v\}_{\p1}=\sum_{j\in\mathbf{Z}} \{i_{j,n}(\check{z}), \phi_v(p^{-j})\}_{\rm dif}.$$
 \end{theorem}

\begin{proof} Notons d\'{e}j\`{a} qu'un tel $n$ existe bien (prop. \ref{picp}).
  Soit $F(\check{z})$ la somme dans le terme de droite. Puisque $\phi_v$ est \`{a} support compact dans $\qpet$, il n'y a qu'un nombre fini
  de termes non nuls dans cette somme (et cela uniform\'{e}ment en $\check{z}$).
  Comme $z\to i_{j,n}(z)$ sont continues
  (car ${\rm Res}_{\zp}$ et $\varphi^{-n}$ le sont),
  on en d\'{e}duit que $F$ d\'{e}finit une forme lin\'{e}aire continue sur
  $ (\Pi^{\rm an})^*$. Mais
  $\Pi^{\rm an}$ est r\'{e}flexif (c'est un espace de type compact), donc il existe
  $v_1\in \Pi^{\rm an}$ tel que $F(\check{z})=\{\check{z}, v_1\}_{\p1}$ pour tout
  $\check{z}\in (\Pi^{\rm an})^*$. On va montrer que
  $v_1=v$, ce qui permettra de conclure. Rappelons que l'on dispose d'une inclusion 
  $\iota: \tilde{\check{D}}^+\to \Pi^*$. Son image est dense d'après le corollaire III.22 de \cite{CD}, donc (par le théorème de Hahn-Banach)
  il suffit de montrer que
  $F(\iota(\check{z}))=\{\iota(\check{z}), v\}_{\p1}$ pour tout $\check{z}\in \tilde{\check{D}}^+$. 
   
   Fixons donc $\check{z}\in\tilde{\check{D}}^+$.
   On a ${\rm Res}_{\zp}\left( \left(\begin{smallmatrix} p^n & 0 \\0 & 1\end{smallmatrix}\right)\iota(\check{z}) \right)=
 \varphi^n(T_n(\check{z}))$ par définition de l'application $\iota$, et en appliquant $\varphi^{-n}$
  on obtient 
   $$i_{0,n}(\iota(\check{z}))=T_n(\check{z})\in \check{D}_{{\rm dif},n}^+.$$
 Puisque $\phi_v(p^{-j})\in
   D_{{\rm dif},n}/ D_{{\rm dif},n}^+$, un retour \`{a} la d\'{e}finition de $\{\,\,,\,\}_{\rm dif}$ combin\'{e} avec l'égalité ci-dessus 
 montrent que
     $$\{i_{j,n}(\iota(\check{z})), \phi_v(p^{-j})\}_{\rm dif}=\{T_n(\varphi^{-j}(\check{z})), \phi_v(p^{-j})\}_{\rm dif}=
     \{\varphi^{-j}(\check{z}), \phi_v(p^{-j})\}_{\rm dif}.$$
   Comme $v\in \Pi_c^{P-\rm fini}$, la preuve de la proposition \ref{ufinicompact} montre l'existence
   d'entiers $a,m,k$ tels que $v$ ait un rel\`{e}vement $z$ \`{a} $\left(\frac{\varphi^a(T)}{\varphi^m(T)}\right)^k\tilde{D}^+$.
   Comme $\varphi^{-j}(z)-\phi_v(p^{-j})
    \in \tilde{D}_{\rm dif}^+$, on a $\{\varphi^{-j}(\check{z}), \phi_v(p^{-j})\}_{\rm dif}=
    \{\varphi^{-j}(\check{z}), \varphi^{-j}(z)\}_{\rm dif}$. Le th\'{e}or\`{e}me est donc une cons\'{e}quence de la
    proposition \ref{reciproc}.

\end{proof}

\begin{corollary}\label{pialg}
  Soit $D\in \mathcal{C}_L(\delta)$ tel que les facteurs de Jordan-Hölder de $D$
  soient de dimension $\geq 2$, et les poids de Hodge-Tate de $D$ soient $\leq 1$. 
  On suppose que $w(\delta)=1-k$ et on pose $\Pi=\Pi_{\delta}(D)$.
Alors $\Pi_{c,k}^{P-\rm alg}$ est un sous-espace de $\Pi^{\rm an}$ 
   tué par $(u^+)^k$ et par $\prod_{i=0}^{k-1} (h-2i+k-1)$.
\end{corollary}
   
   \begin{proof}
    Le fait que $\Pi_{c,k}^{P-\rm alg}$ est un sous-espace de $\Pi^{\rm an}$ 
   tué par $(u^+)^k$ découle du théorème \ref{dualKir} et du lemme \ref{uplus}. 
   Pour vérifier la dernière assertion, notons $a^+\in \mathfrak{gl}_2$ l'action
   infinitésimale de $\left(\begin{smallmatrix} \zpet & 0\\0 & 1\end{smallmatrix}\right)$. 
   Comme $\Pi$ a pour caractère central $\delta$ et $w(\delta)=1-k$, on a
   $a^+=\frac{h+1-k}{2}$ sur $\Pi^{\rm an}$. Par définition de 
    $\Pi_{c,k}^{P-\rm alg}$, chacun de ses vecteurs est tué par
    $\prod_{i=0}^{k-1} (a^+-(i+1-k))$. Le résultat s'en déduit.
    
   \end{proof}
   
    Nous terminons ce $\S$ par un calcul qui nous sera utile plus loin:
    
    \begin{lemma}\label{ijn}
     Soient $\mathcal{D}\in \mathcal{C}_L(\delta)$ et $k\in\mathbf{N}^*$ tels que $u^{-}(\check{z})=\frac{\nabla(k-\nabla)(\check{z})}{t}$
    pour tout $\check{z}\in \check{\mathcal{D}}_{\rm rig}$. Alors 
       $$i_{j,n}((u^{-})^k \check{z})=(-p^j)^k \frac{\nabla_{2k}(i_{j,n}(\check{z}))}{t^k}$$
       pour tous $\check{z}\in (\Pi_{\delta}(\mathcal{D})^{\rm an})^*$, $n\geq m(\mathcal{D})$ et $j\in\mathbf{Z}$. 
       
    \end{lemma}
   
   \begin{proof}
    On fixe $n\geq m(\mathcal{D})$, $j\in\mathbf{Z}$ et $\check{z}\in (\Pi_{\delta}(\mathcal{D})^{\rm an})^*$, et on pose
   $$x={\rm Res}_{\zp} \left( \left(\begin{smallmatrix} p^{n-j} & 0\\0 & 1\end{smallmatrix}\right)\check{z}\right)\in\mathcal{\check{D}}_{\rm rig},$$ de telle sorte 
   que $i_{j,n}(\check{z})=\varphi^{-n}(x)$. Un petit calcul montre que l'on a 
   $$ \left(\begin{smallmatrix} a & 0\\0 & 1\end{smallmatrix}\right)\circ u^{-}=\frac{1}{a} u^{-}\circ  \left(\begin{smallmatrix} a & 0\\0 & 1\end{smallmatrix}\right),$$
   donc 
   $$i_{j,n}((u^{-})^k \check{z})=\varphi^{-n}( p^{k(j-n)} {\rm Res}_{\zp} ((u^{-})^k  \left(\begin{smallmatrix} p^{n-j} & 0\\0 & 1\end{smallmatrix}\right)\check{z}))=$$
   $$=p^{k(j-n)}\varphi^{-n}( (u^{-})^k(x)),$$
   la dernière égalité étant une conséquence du fait que $u^{-}$ commute à ${\rm Res}_{\zp}$ \cite[prop. VI.8]{CD}. 
   Le résultat découle alors du lemme \ref{long} et du fait que $\varphi^{-n}$ commute à $\nabla$ et envoie $t$ sur $t/p^n$.
   
   \end{proof}

\subsection{Un résultat de densité}

   Comme application du théorème \ref{dualKir}, nous allons démontrer le résultat suivant:

   \begin{theorem}\label{dense}
             Soit $\Pi\in {\rm Rep}_L(\delta)$ dont tous les facteurs de Jordan-Hölder sont supersinguliers. Alors 
              $\Pi_c^{P-\rm fini}$ est un sous-espace dense de $\Pi$, contenu dans $\Pi^{\rm an}$ et stable sous l'action de $B$.
           \end{theorem}
           
      \begin{proof}   Soit $D$ le dual de Cartier de $D(\Pi)$, de telle sorte que $\Pi\simeq \Pi_{\delta}(D)\simeq \tilde{D}/\tilde{D}^+$ (th. \ref{CD} et 
       prop. \ref{Bisom}).  
     Le fait que $\Pi_c^{P-\rm fini}$ est contenu dans $\Pi^{\rm an}$ et est stable par $B$ découle
      du théorème \ref{dualKir} et du corollaire \ref{picp}. Il reste donc à prouver qu'il est dense dans $\Pi$.
     Supposons que ce n'est pas le cas, donc il existe $\check{z}\in \Pi^*$
       tel que $\{\check{z}, v\}_{\p1}=0$ pour tout $v\in \Pi_c^{P-\rm fini}$. On a une injection $\Pi^*\subset (\Pi^{\rm an})^*$
       déduite de la densité de $\Pi^{\rm an}$ dans $\Pi$ \cite{STInv}. Le théorème \ref{dualKir} et le corollaire \ref{picp}
       montrent que $i_{j,n}(\check{z})$ est orthogonal à $D_{{\rm dif},n}/D_{{\rm dif},n}^+$ pour $j\in\mathbf{Z}$ et 
       $n\geq m(D)$. Puisque $\{\,\,,\,\}_{\rm dif}$ induit un accouplement parfait entre 
       $\check{D}_{{\rm dif},n}^+$ et $D_{{\rm dif},n}/D_{{\rm dif},n}^+$, il s'ensuit que 
       $i_{j,n}(\check{z})=0$ pour $j\in\mathbf{Z}$ et $n\geq m(D)$. En revenant à la définition des applications 
       $i_{j,n}$ et en utilisant l'injectivité de $\varphi^{-n}$, on obtient ${\rm Res}_{\qp}(\check{z})=0$. La proposition 
       III.23 de \cite{CD} montre que sous les hypothèses du théorème l'application ${\rm Res}_{\qp}$ est injective
       sur $\Pi^*\subset \check{D}\boxtimes_{\delta^{-1}}\p1$, ce qui permet de conclure. 
      
      \end{proof}

\section{Etude du module $\check{E}_{{\rm  dif}, n}^+$}

  \quad Dans ce chapitre on fixe $D\in\Phi\Gamma^{\rm et}(\mathcal{E})$ de dimension $2$, de de Rham, à poids de 
  Hodge-Tate $1$ et $1-k$, avec $k\in\mathbf{N}^*$. On note $\delta=\chi^{-1}\det V(D)$, de telle sorte que $w(\delta)=1-k$. On fixe une extension non scindée
   $0\to D\to E\to D\to 0$ telle que $E$ soit de Hodge-Tate et $E\in\mathcal{C}_L(\delta)$. 
  On décrit la structure du module $\check{E}_{{\rm  dif}, n}^+$ et on introduit un certain nombre de modules qui permettent de 
  contrôler $\Pi_{\delta}(E)^{\rm alg}$.
    Pour simplifier, on note $R_n=L_n[[t]]$.  Les entiers 
   $m(E)\geq m(D)$ sont des constantes assez grandes, qui ne dépendent que de $D$ et $E$.

\subsection{Construction d'une base convenable}

   \quad Puisque $\check{D}$ est de de Rham à poids de Hodge-Tate $0$ et $k$, les $L$-espaces vectoriels 
   $\check{D}_{\rm dR}^+$ et $\check{D}_{\rm dR}/ \check{D}_{\rm dR}^+$ sont de dimension $1$. 
    On en déduit\footnote{Choisir une base $e_1$ de $\check{D}_{\rm dR}^+$
   et un élément $x\in \check{D}_{\rm dR}$ tel que l'image de $x$ dans $\check{D}_{\rm dR}/\check{D}_{\rm dR}^+$ en soit une base, et poser 
   $e_2=t^k x$.} l'existence de $e_1, e_2\in \check{D}_{{\rm dif}, m(D)}^+$ tels
   que $$\sigma_a(e_1)=e_1, \quad \sigma_a(e_2)=a^k e_2, \quad \forall a\in \zpet$$
   et tels que $\check{D}_{{\rm dif}, n}^+$ soit un $R_n$-module libre de base $e_1, e_2$ 
   pour tout $n\geq m(D)$. 
   
   \begin{proposition}\label{basedif}
 Il existe $f_1,f_2\in \check{E}_{{\rm dif},m(E)}^+$ tels que

   a) $\check{E}_{{\rm dif},n}^+$ est un $R_n$-module libre de base 
    $e_1,e_2,f_1,f_2$ pour tout
   $n\geq m(E)$.

   b) Il existe $\alpha\in L$ tel que $\sigma_a(f_1)=f_1$ et $\sigma_{a}(f_2)=
   a^k f_2+\alpha a^k \log a\cdot t^k e_1$ pour tout $a\in \zpet$.

\end{proposition}

\begin{proof} Pour simplifier, on note $m=m(E)$, $X=\check{E}_{{\rm dif}, m}^+$, $Y=\check{D}_{{\rm dif}, m}^+$,
 et enfin $a=1+p^{m}$. La suite exacte $0\to D^{]0,r_{m}]}\to E^{]0, r_{m}]}\to D^{]0,r_{m}]}\to 0$
 et la platitude de $R_m$
 sur $\mathcal{E}^{]0,r_{m}]}$ fournissent une suite exacte de $R_m$-modules
 $0\to Y\to X\to Y\to 0$, compatible avec l'action de $\Gamma$. 

   Soient $f_1,f_2\in X$ qui s'envoient sur $e_1, e_2$, respectivement, de telle sorte que
   $e_1,e_2,f_1,f_2$ est une $R_m$-base de $X$. On a $\sigma_{a}(f_1)-f_1\in Y$
   et $\sigma_{a}(f_2)-a^k f_2\in Y$. Si $A,B\in R_m$, un petit calcul montre que
    $$(\sigma_{a}-1)(f_1+Ae_1+Be_2)=(\sigma_{a}(f_1)-f_1)+(\sigma_{a}(A)-A)e_1+
    (a^k\sigma_{a}(B)-B)e_2.$$
    Mais $B\to a^k\sigma_{a}(B)-B$ est surjective sur $R_m$
    et $A\to \sigma_{a}(A)-A$ a pour image l'ensemble des \'{e}l\'{e}ments de $R_m$ de terme constant
    nul (cela d\'{e}coule trivialement du fait que $\sigma_{a}(t)=at$). Ainsi, quitte \`{a} remplacer $f_1$ par $f_1+Ae_1+Be_2$ pour $A,B\in R_m$ convenables,
    on peut supposer que $\sigma_{a}(f_1)-f_1=\alpha\cdot e_1$ pour un $\alpha\in L_{m}$.
    Alors $(\sigma_{a}-1)^2(f_1)=0$ et donc $\nabla^2(f_1)=0$. On en d\'{e}duit que
    l'image de $f_1$ dans $X/tX$ est tu\'{e}e par $\Theta_{{\rm Sen}, \check{E}}^2$ et donc
    aussi par $\Theta_{{\rm Sen},\check{E}}$ (th. \ref{EHT}). Ainsi,
    $\nabla(f_1)=\frac{\alpha}{\log a}e_1\in tX$ et comme
    $e_1\notin tX$, on en d\'{e}duit que $\alpha=0$ et donc $\sigma_{a}(f_1)=f_1$.

       Pour $f_2$, l'argument est similaire: quitte à modifier $f_2$ par un élément de $Y$, on peut supposer que $$(\sigma_{a}-a^k)(f_2)=
       \alpha a^k \log a\cdot t^k e_1+\beta e_2,$$ avec $\alpha,\beta\in L_{m}$. Alors 
       $(\sigma_a-a^k)^2(f_2)=0$, donc $(\nabla-k)^2(f_2)=0$ et comme avant 
       $(\nabla-k)(f_2)\in tX$. Puisque
            $$(\nabla-k)(f_2)=\frac{\log (1+\frac{\sigma_a-a^k}{a^k})}{\log a} (f_2)=\alpha t^k e_1+\frac{\beta e_2}{a^k\log a}\equiv \frac{\beta}{a^k\log a}e_2\pmod {tX}$$
            et $e_2\notin tX$,
            on obtient $\beta=0$. 
       
      Pour l'instant on a obtenu $f_1,f_2$ tels que $\sigma_{a}(f_1)=f_1$ et
       $\sigma_{a}(f_2)=
   a^k f_2+\alpha a^k \log a\cdot t^k e_1$ pour $a=1+p^m$. Un calcul immédiat montre que 
   $\sigma_{b}(f_1)=f_1$ et $\sigma_b(f_2)=b^k f_2+\alpha b^k \log b\cdot t^k e_1$ quand 
   $b\in \{a, a^2,...\}$, donc pour tout $b\in 1+p^m\zp$. 
    On va enfin modifier $f_1$ et $f_2$ pour que ceci soit valable pour tout $b\in\zpet$. 
    
    Soit
  $\Gamma_m=\chi^{-1}(1+p^m\zp)$. Puisque $f_1$ s'envoie sur $e_1$ et $e_1$ est $\Gamma$-invariant, on a $(\sigma_b-1)f_1\in Y^{\Gamma_m}$
  pour tout $b\in\zpet$. Un calcul imm\'{e}diat montre que $Y^{\Gamma_m}=L_m e_1$, 
  ce qui fournit $c_b\in L_m$ tel que $(\sigma_b-1)f_1=c_b e_1$. L'application $b\to c_b$ définit
  donc un $1$-cocycle de $\Gamma/\Gamma_m$ dans $L_m$. Le théorème de Hilbert $90$ permet de conclure
  que l'on peut remplacer $f_1$ par $f_1-xe_1$ pour un $x\in L_m$ convenable, pour le rendre $\Gamma$-invariant.
  Pour $f_2$ la situation est un peu plus compliquée, car dans ce cas $(\sigma_b-b^k)(f_2)\in Y_m^{\Gamma_m=\chi^k}$
  et $Y_m^{\Gamma_m=\chi^k}$ est un $L_m$-module libre de rang $2$, une base étant $e_2$ et $t^k e_1$. On peut donc écrire 
  $$(\sigma_b-b^k)(f_2)=b^k c'_b e_2+b^k (\alpha\log b+c_b)\cdot t^k e_1$$ avec 
  $c_b, c'_b: \Gamma/\Gamma_m\to L_m$. Un calcul fastidieux montre que 
  $c_b, c_b'$ sont des $1$-cocycles, ce qui permet de conclure comme ci-dessus.

   \end{proof}

\begin{lemma}\label{alphanul}
 Avec les notations de la proposition \ref{basedif}, on a $\alpha=0$ si et seulement
 si $E$ est de de Rham. 
\end{lemma}

\begin{proof}

  $E$ est de de Rham si et seulement si $\check{E}$
 l'est, et cela arrive si et seulement si $(\check{E}_{{\rm dif},n})^{\Gamma}$ est un $L$-espace vectoriel de dimension $4$ pour un/tout $n\geq m(E)$,
 puisque $D_{\rm dR}(V(\check{E}))\simeq (\check{E}_{{\rm dif},n})^{\Gamma}$.
 Or, en utilisant les formules
 explicites fournies par la proposition \ref{basedif}, on v\'{e}rifie sans mal que $(\check{E}_{{\rm dif},n})^{\Gamma}$ est le $L$-espace vectoriel de base
$e_1,\frac{e_2}{t^k}, f_1, \frac{f_2}{t^k}$ si $\alpha=0$ et le
$L$-espace vectoriel de base $e_1,\frac{e_2}{t^k}, f_1$ sinon.
\end{proof}

\begin{proposition}\label{nouvellebase}

    Il existe $e_1', e_2', f_1', f_2'\in E_{{\rm dif}, m(E)}^{+}$ tels que

    a) $e_1', e_2'$ est une $R_n$-base de $D_{{\rm dif},n}^+$ et
    $e_1', e_2', f_1', f_2'$ est une $R_n$-base de $E_{{\rm dif},n}^+$ pour $n\geq m(E)$.

   b) On a $\sigma_{a}(e_1')=a^{1-k}e_1'$, $\sigma_a(e_2')=ae_2'$, $\sigma_{a}(f_1')=a^{1-k}f_1'$ et enfin
    $\sigma_{a}(f_2')=af_2'+\beta a\log a\cdot t^k e_1'$ pour tout $a\in\zpet$, pour un certain $\beta\in L$.

\end{proposition}

\begin{proof} La preuve est identique à celle de la proposition \ref{basedif}, en 
travaillant avec $D\otimes \chi^{k-1}$ et $E\otimes \chi^{k-1}$ au lieu de $\check{D}$ et $\check{E}$. 
     
 \end{proof}

\subsection{Décomposition de $\Pi_c^{P-\rm alg}$}

        \quad Si $M$ est un $\Gamma$-module, on note $M(i)$ le module $M\otimes \chi^{i}$ obtenu en tordant l'action de 
        $\Gamma$ par $\chi^i$. 
       
       \begin{proposition}\label{tatetwist}
       
       a) L'application $R_n\to D_{{\rm dif}, n}$ qui à $A$ associe $\frac{A}{t^k} e_2'$ induit un isomorphisme de 
       $R_n[\Gamma]$-modules 
       $$ R_n/t^kR_n\simeq \left( D_{{\rm dif}, n}/ D_{{\rm dif}, n}^+\right)^{\mu_{n,k}=0}(k-1).$$
       
       b) L'application $R_n\oplus R_n\to E_{{\rm dif}, n}$ qui à $(A,B)$ associe $\frac{A}{t^k}e_2'+\frac{B}{t^k}f_2'$ induit un isomorphisme de
       $R_n[\Gamma]$-modules 
        $$ R_n/t^kR_n\oplus R_n/t^k R_n\simeq \left( E_{{\rm dif}, n}/ E_{{\rm dif}, n}^+\right)^{\mu_{n,k}=0}(k-1).$$

       \end{proposition}
       
       \begin{proof}
       On ne traite que le b), qui est plus délicat. Notons $a=1+p^m$, fixons $0\leq i<k$ et considérons un élément 
       $Ae_1'+Be_2'+Cf_1'+Df_2'\in E_{{\rm dif}, n}$ dont la classe appartient à 
       $\left( E_{{\rm dif}, n}/ E_{{\rm dif}, n}^+\right)^{\sigma_a=a^{i+1-k}}$. En utilisant les formules 
       de la proposition \ref{nouvellebase}, cette appartenance se traduit par 
       $$\sigma_a(B)-a^{i-k}B\in R_n, \quad \sigma_a(C)-a^i C\in R_n, \quad \sigma_a(D)-a^{i-k}D\in R_n,$$
       $$ (\sigma_a(A)-a^iA)a^{1-k} +\beta a\log a\cdot t^k \sigma_{a}(D)\in R_n.$$
       Les trois premières conditions entraînent $B\in \frac{1}{t^{k-i}} L_n+R_n$, 
       $D\in \frac{1}{t^{k-i}} L_n+R_n$ et $C\in R_n$. Puisque $t^k D\in R_n$, la dernière condition entraîne 
       $A\in R_n$. En utilisant ces observations et le lemme des noyaux, on obtient que
       $\left( E_{{\rm dif}, n}/ E_{{\rm dif}, n}^+\right)^{\mu_{n,k}=0}$ est l'ensemble des classes $\overline{Be_2'+Df_2'}$ avec 
       $B,D\in t^{-k}R_n$. Cela montre déjà que l'application  $F: (A,B)\to \frac{A}{t^k}e_2'+\frac{B}{t^k}f_2'$
       induit un isomorphisme de $R_n$-modules $ R_n/t^kR_n\oplus R_n/t^k R_n\simeq \left( E_{{\rm dif}, n}/ E_{{\rm dif}, n}^+\right)^{\mu_{n,k}=0}$.
       En utilisant à nouveau les formules de la proposition \ref{nouvellebase}, on vérifie que $F\circ\sigma_a=a^{k-1}\sigma_a\circ F$ pour tout 
     $a\in \zpet$, donc $F\otimes \chi^{k-1}$ est un isomorphisme de $R_n[\Gamma]$-modules, ce qui permet de conclure.
       
       \end{proof}

   La suite exacte $0\to D^{]0, r_n]}\to E^{]0, r_n]}\to D^{]0, r_n]}\to 0$ induit (par localisation en $\varepsilon^{(n)}-1$) des suites exactes de $R_n[\Gamma]$-modules
    $0\to D_{{\rm dif}, n}^+\to E_{{\rm dif}, n}^+\to D_{{\rm dif}, n}^+\to 0$ et  $0\to D_{{\rm dif}, n}\to E_{{\rm dif}, n}\to D_{{\rm dif}, n}\to 0$ 
    pour $n\geq m(E)$. Le lemme du serpent fournit donc une suite exacte de $R_n[\Gamma]$-modules 
    $$0\to D_{{\rm dif}, n}/ D_{{\rm dif}, n}^+\to E_{{\rm dif}, n}/ E_{{\rm dif}, n}^+\to D_{{\rm dif}, n}/ D_{{\rm dif}, n}^+\to 0.$$
    
    \begin{lemma}
     Pour tout $n\geq m(E)$, la suite exacte précédente induit une décomposition de $R_n[\Gamma]$-modules 
     $$\left(E_{{\rm dif},n}/E_{{\rm dif},n}^+\right)^{\mu_{n,k}=0}\simeq \left(D_{{\rm dif},n}/D_{{\rm dif},n}^+\right)^{\mu_{n,k}=0}
     \oplus \left(D_{{\rm dif},n}/D_{{\rm dif},n}^+\right)^{\mu_{n,k}=0}.$$
    \end{lemma}
    
    \begin{proof}
    C'est une conséquence directe de la proposition \ref{tatetwist}. 
    
    \end{proof}

    \begin{proposition}\label{picscindee}
     La suite exacte $0\to \Pi_{\delta}(D)\to \Pi_{\delta}(E)\to\Pi_{\delta}(D)\to 0$ induit une suite exacte scindée de 
     $B$-modules $$0\to \Pi(D)^{P-\rm alg}_c\to \Pi(E)^{P-\rm alg}_c\to \Pi(D)_c^{P-\rm alg}\to 0.$$
     
    \end{proposition}
    
    \begin{proof}    
      Cela découle du lemme précédent et de la proposition \ref{pipalg}.
   
\end{proof}

\subsection{Les modules $O_n(\mathcal{D})$ et $M_n(\mathcal{D})$}

   \quad Dans ce $\S$ on introduit certains sous-modules de $\check{E}_{{{\rm dif},n}}^+$ qui joueront un rôle crucial dans les arguments
   de dualité dont on aura besoin pour contrôler $\Pi_{\delta}(E)^{\rm alg}$.  On garde les hypothèses du début de ce chapitre et on rappelle que  
   $$\mu_{n,k}=\prod_{i=0}^{k-1} (\sigma_{1+p^n}-(1+p^n)^{i+1-k})\in O_L[\Gamma].$$ 

\begin{definition}\label{orthoon}
Si $\mathcal{D}\in\{D, E\}$ et si $n\geq m(\mathcal{D})$, on note (voir le $\S$\ref{dualitefond} pour $\{\,\, ,\,\}_{\rm dif}$)
$$ O_{n}(\mathcal{D})=\left \{ \check{z}\in \check{\mathcal{D}}_{{{\rm dif},n}}^+ |\quad  \{\check{z}, z\}_{\rm dif}=0 \quad \forall 
z\in \left(\mathcal{D}_{{\rm dif},n}/ \mathcal{D}_{{\rm dif},n}^+\right)^{\mu_{n,k}=0}\right \}.$$

\end{definition}

  \begin{lemma}\label{calculon}
  Pour tout $n\geq m(\mathcal{D})$ on a 
   $$O_n(\mathcal{D})=\bigcap_{i=0}^{k-1}(\sigma_{1+p^n}-(1+p^n)^i)\check{\mathcal{D}}_{{\rm dif},n}^+.$$

  \end{lemma}

\begin{proof} Notons $a=1+p^n$ et $\sigma=\sigma_{a}$.
 Soit $0\leq i\leq k-1$. Puisque $\{\, \,,\,\}_{\rm dif}$ est parfait
 et $\Gamma$-invariant, 
 $\left(\mathcal{D}_{{\rm dif},n}/ \mathcal{D}_{{\rm dif},n}^+\right)^{\sigma-a^{i+1-k}=0}$ est en dualit\'{e} parfaite
 avec $\frac{\check{\mathcal{D}}_{{\rm dif},n}^+}{\sigma^{-1}-a^{i+1-k}}$. Cette observation combinée avec le c) de la remarque \ref{ddifalg} permettent de conclure que $$O_n(\mathcal{D})=\bigcap_{i=0}^{k-1}
 (\sigma^{-1}-a^{i+1-k})\check{\mathcal{D}}_{{\rm dif},n}^+=\bigcap_{i=0}^{k-1}(\sigma-a^i)\check{\mathcal{D}}_{{\rm dif},n}^+.$$
 \end{proof}

  \begin{proposition}\label{mn}
  a) Pour tout $n\geq m(D)$, $O_n(D)$ est un $R_n$-module libre de base $t^k e_1$ et  $e_2$.
  
  b) Pour tout $n\geq m(E)$, $O_n(E)$ est un $R_n$-module libre de base  $t^ke_1, t^kf_1, e_2, f_2$.

    \end{proposition}
    
    \begin{proof}  Nous allons traiter seulement le b), le a) étant similaire, et plus simple. 
    Fixons $n\geq m(E)$ et notons $a=1+p^n$ et $\sigma=\sigma_a$. Si $A,B,C,D\in R_n$, un calcul imm\'{e}diat montre que pour $0\leq i<k$ on a
    $$(\sigma-a^i)( Ae_1+Be_2+Cf_1+Df_2)=((\sigma-a^i)A+\alpha a^k t^k\log a\cdot \sigma(D))e_1+$$
    $$(a^k\sigma-a^i)B\cdot e_2+(\sigma-a^i)C\cdot f_1+(a^k\sigma-a^i)D\cdot f_2.$$
    Or $B\to (a^k\sigma-a^i)B$ est bijective sur $R_n$ et l'image de $A\to (\sigma-a^i)A$ est l'ensemble
 des s\'{e}ries $f\in R_n$ dont le coefficient de $t^i$ est nul (utiliser le fait que $\sigma(t)=at$).
    On en d\'{e}duit  que $(\sigma-a^i)\check{E}_{{\rm dif},n}^+$ est l'ensemble des combinaisons
 $A_1e_1+B_1e_2+C_1f_1+D_1f_2$ telles que les coefficients de $t^i$ dans $A_1$ et $C_1$ soient nuls.
On conclut en prenant l'intersection sur $i\in \{0,1,...,k-1\}$ et en utilisant le lemme \ref{calculon}.

       \end{proof}
           
    \quad  Rappelons que $\nabla_{2k}=\prod_{i=0}^{2k-1}(\nabla-i)$. 
  Le lemme suivant est plus ou moins implicite dans la preuve de \cite[lemme 3]{Annalen}:
  
    \begin{lemma}\label{useless}
     Soient $B\in R_n$ et $P=\prod_{i=0}^{2k-1}(X-i)$. Alors $P'(\nabla)(B)\in t^{2k}R_n$ si et seulement
     si $B\in t^{2k}R_n$.
    \end{lemma}

     \begin{proof}
     Si $B=\sum_{j\geq 0} \alpha_jt^j$, alors on a des \'{e}galit\'{e}s dans $R_n/t^{2k}R_n$
   $$ \sum_{j=0}^{2k-1}\frac{\nabla_{2k}}{\nabla-j} B=\sum_{s=0}^{2k-1}\alpha_s\cdot \sum_{j=0}^{2k-1}\frac{\nabla_{2k}}{\nabla-j}(t^s)$$
   $$=\sum_{s=0}^{2k-1} \alpha_s\cdot \frac{\nabla_{2k}}{\nabla-s}(t^s)=\sum_{s=0}^{2k-1} (-1)^{2k-1-s}s!\cdot (2k-s-1)!\alpha_s\cdot t^s.$$
  On a donc $P'(\nabla)(B)\in t^{2k}R_n$ si et seulement si $\alpha_s=0$ pour tout $0\leq s\leq 2k-1$, d'où le résultat.
     \end{proof}

\begin{definition}
Si $\mathcal{D}\in\{D, E\}$ et si $n\geq m(\mathcal{D})$, on pose $$M_n(\mathcal{D})=\{z\in \check{\mathcal{D}}_{{\rm dif},n}^+|
  \nabla_{2k}(z)\in t^{k}O_n(\mathcal{D})\}.$$ 
\end{definition}

\begin{proposition} \label{mne}  
 
 a) $M_n(D)=\check{D}_{{\rm dif},n}^+$ pour $n\geq m(D)$.
 
 b) Si $E$ est de de Rham, alors $M_n(E)=\check{E}_{{\rm dif},n}^+$ pour tout $n\geq m(E)$. Si ce n'est pas le cas, alors 
 $M_n(E)$ est un $R_n$-module libre de base $e_1,f_1, e_2,t^kf_2$.
\end{proposition}

 \begin{proof} Nous ne traiterons que le b), qui est plus technique. En passant \`{a} la limite (pour $a\to 1$) dans la proposition \ref{basedif} on obtient $$\nabla(e_1)=0, \quad \nabla(f_1)=0, \quad \nabla\left(\frac{e_2}{t^k}\right)=0, \quad
   \nabla\left(\frac{f_2}{t^k}\right)=\alpha\cdot e_1.$$
   On en déduit,
   pour $B\in L_n((t))$ et $P\in L_n[X]$, les \'{e}galit\'{e}s
    $P(\nabla)(B\cdot x)=P(\nabla)(B)\cdot x$ si $x\in\{e_1, f_1\}$,
    $$\quad P(\nabla)(B\cdot e_2)=P(\nabla)\left( t^kB\cdot \frac{e_2}{t^k}\right)=\frac{P(\nabla)(t^kB)}{t^k}\cdot e_2=P(\nabla+k)(B)\cdot e_2$$
 et enfin, avec un calcul semblable à celui pour $P(\nabla)(B\cdot e_2)$, 
     $$P(\nabla)(B\cdot f_2)=P(\nabla+k)(B)\cdot f_2+\alpha P'(\nabla+k)(B)\cdot t^ke_1.$$

  Soit $P=\prod_{i=0}^{2k-1}(X-i)$, de telle sorte que $P(\nabla)=\nabla_{2k}$. On v\'{e}rifie
     sans mal que $P(\nabla)(R_n)\subset t^{2k}R_n$ et que $P(\nabla+k)(R_n)\subset
     t^k R_n$. En utilisant la proposition \ref{mn}, on en déduit que 
      $Ae_1+Be_2+Cf_1+Df_2\in \check{E}_{{\rm dif}, n}^+$ est dans $M_n(E)$ 
     si et seulement si $\alpha P'(\nabla+k)(D)\in t^k R_n$. On conclut en utilisant les lemmes \ref{alphanul} et
\ref{useless} et en remarquant que la condition $\alpha P'(\nabla+k)(D)\in t^k R_n$ équivaut à 
$\alpha P'(\nabla)(t^k D)\in t^{2k} R_n$. 
  \end{proof}

\section{Preuve des résultats principaux}

   \quad On fixe dans ce chapitre $D\in\Phi\Gamma^{\rm et}(\mathcal{E})$ de dimension $2$, tel que $V(D)$ soit de de Rham,
   à poids de Hodge-Tate distincts. On se donne une 
 suite exacte non scindée $0\to D\to E\to D\to 0$ et on suppose que $E\in \mathcal{C}_L(\delta)$, avec 
 $\delta=\chi^{-1}\det V(D)$. On ne fait aucune hypothèse sur $p$. 

    \begin{theorem}\label{main1}
       Avec les hypothèses ci-dessus, les assertions suivantes sont équivalentes:
       
       a) $\Pi_{\delta}(E)^{\rm alg}\ne \Pi_{\delta}(D)^{\rm alg}$.
       
       b) $V(E)$ est de de Rham.
       
       c) $\Pi_{\delta}(E)^{\rm alg}$ est dense dans $\Pi_{\delta}(E)$.

      \end{theorem}
  
    Avant de passer à la preuve du théorème \ref{main1}, expliquons pourquoi il entraîne le théorème principal \ref{main}.
    On se place dans le contexte du théorème \ref{main}, pour lequel on renvoie à l'introduction. Soient
    $D$ et $E$ les duaux de Cartier de $D(\Pi)$ et $D(\Pi_1)$. D'après Paskunas \cite{Pa}, $D$ est de dimension $2$, absolument irréductible et 
    $\delta$ est le caractère central de $\Pi$. D'après le d) du théorème \ref{CD}, on a des isomorphismes canoniques
    $\Pi\simeq \Pi_{\delta}(D)$ et $\Pi_1\simeq \Pi_{\delta}(E)$. 
    Puisque $\Pi^{\rm alg}\ne 0$, un théorème de Colmez (\cite[th. 0.20]{Cbigone} ou \cite[th.4]{Annalen}) 
    montre que $D$ est de de Rham, à poids de Hodge-Tate distincts. Les hypothèses du théorème \ref{main1} sont 
    donc satisfaites, ce qui permet de conclure. 
    
      Revenons maintenant à la preuve du théorème \ref{main1}. Quitte à faire une torsion par un caractère algébrique, on peut supposer que les poids de 
  Hodge-Tate de $D$ sont $1$ et $1-k$, avec $k\in \mathbf{N}^*$. On note $\Pi_c^{P-\rm alg}$ au lieu de $\Pi_{c,k}^{P-\rm alg}$ et pareil pour $\Pi_1$.
  La proposition \ref{HTalg} nous permet de supposer que $E$ est de Hodge-Tate.

  \begin{proposition}\label{crux}
   Si $\mathcal{D}\in \{D,E\}$, alors les assertions suivantes sont équivalentes:
   
   a) $\mathcal{D}$ est de de Rham.
   
   b) $(u^{-})^k$ tue $\Pi_{\delta}(\mathcal{D})_c^{P-\rm alg}$.
   
   c) $\Pi_{\delta}(\mathcal{D})_c^{P-\rm alg}\subset \Pi_{\delta}(\mathcal{D})^{\rm alg}$.
   
  \end{proposition}

\begin{proof} On note pour simplifier $\Pi=\Pi_{\delta}(\mathcal{D})$ et, pour $n\geq m(\mathcal{D})$, on pose 
$$N_n=\varphi^{-n}({\rm Res}_{\zp}((\Pi^{\rm an})^*)).$$
Nous aurons besoin de quelques préliminaires:
   
  \begin{lemma}\label{key}
    $(u^{-})^k$ tue $\Pi_c^{P-\rm alg}$ si et seulement si $N_n\subset M_n(\mathcal{D})$ pour tout $n\geq m(\mathcal{D})$.
  
  \end{lemma}
  
  \begin{proof} Par Hahn-Banach et 
  puisque $\{\,\,,\,\}_{\p1}$ est $G$-équivariant, $(u^{-})^k$ tue $\Pi_c^{P-\rm alg}$ 
si et seulement si $(u^{-})^k (\Pi^{\rm an})^*$ est orthogonal à $\Pi_c^{P-\rm alg}$. 
D'après le théorème \ref{dualKir} cela arrive si et seulement si 
$$\sum_{j\in\mathbf{Z}} \{i_{j,n}((u^{-})^k\check{z}), \phi_v(p^{-j})\}_{\rm dif}=0$$
pour tous $\check{z}\in (\Pi^{\rm an})^*$, $v\in \Pi_c^{P-\rm alg}$ et $n\geq m(\mathcal{D})$
assez grand pour que $\phi_v(p^j)\in D_{{\rm dif},n}/D_{{\rm dif},n}^+$ pour tout 
$j\in\mathbf{Z}$.
Le corollaire \ref{blah} montre que cela équivaut encore à l'égalité
$$\sum_{j\in\mathbf{Z}} \{i_{j,n}((u^{-})^k\check{z}), u_j\}_{\rm dif}=0$$
pour tous $\check{z}\in (\Pi^{\rm an})^*$, $(u_j)_{j\in\mathbf{Z}}\in \oplus_{j\in\mathbf{Z}} (D_{{\rm dif}, n}/D_{{\rm dif},n}^+)^{\mu_{n,k}=0}$
et $n\geq m(\mathcal{D})$. Par définition de 
$O_n(\mathcal{D})$, cela revient à $i_{j,n}((u^{-})^k\check{z})\in O_n(\mathcal{D})$ pour 
$j\in\mathbf{Z}$, $\check{z}\in (\Pi^{\rm an})^*$, $n\geq m(\mathcal{D})$. Enfin, d'après 
le lemme \ref{ijn} ceci équivaut à $i_{j,n}(\check{z})\in M_n(\mathcal{D})$ pour 
tout $\check{z}\in (\Pi^{\rm an})^*$, $j\in\mathbf{Z}$ et $n\geq m(\mathcal{D})$.
On conclut en revenant à la définition des applications $i_{j,n}$. 
  
  \end{proof}

Puisque 
   $(\Pi^{\rm an})^*\subset \check{D}^{]0,r_{m(\mathcal{D})}]}\boxtimes_{\delta^{-1}}\p1$, on a 
   $N_n\subset \check{\mathcal{D}}_{{\rm dif},n}^+$ pour $n\geq m(\mathcal{D})$.
    L'implication a) entraîne b) découle donc 
  de la proposition \ref{mne} et du lemme \ref{key}.
  
  Supposons que b) est vraie. On déduit alors du corollaire \ref{pialg} que 
   tout vecteur de 
  $\Pi_c^{P-\rm alg}$ est tué par un idéal de codimension finie dans $U(\mathfrak{sl}_2)$, et donc est 
  ${\rm SL}_2(\qp)$-algébrique. Puisque $\delta$ est localement algébrique, on obtient 
  $\Pi_c^{P-\rm alg}\subset \Pi^{\rm alg}$, d'où le c). 
  
  Enfin, si c) est vraie, la proposition \ref{pilc}
  montre que $(u^{-})^k$ tue $\Pi^{\rm alg}$ et donc $(u^{-})^k$ tue $\Pi_c^{P-\rm alg}$. 
  Le lemme \ref{key} permet de conclure que $N_n\subset M_n(\mathcal{D})$ 
  pour $n\geq m(\mathcal{D})$. Nous avons besoin du lemme suivant:

\begin{lemma} \label{engendre}
 $N_n$ engendre $\check{\mathcal{D}}_{{\rm dif},n}^+$ en tant que $L_n[[t]]$-module.

\end{lemma}

\begin{proof} Puisque $\Pi^{\rm an}$ est dense dans $\Pi$ \cite[th. 7.1]{STInv}, on a une inclusion 
$\Pi^*\subset (\Pi^{\rm an})^*$, donc il suffit de vérifier que $\varphi^{-n}({\rm Res}_{\zp}(\Pi^*))$ engendre 
$\check{\mathcal{D}}_{{\rm dif},n}^+$. Il découle des corollaires III.24 de \cite{CD}, II.5.12 
et II.7.2 de \cite{Cmirab} que ${\rm Res}_{\zp}(\Pi^*)$ contient une base de $D^{]0,r_n]}$ pour $n\geq m(\mathcal{D})$ (cela peut demander d'augmenter $m(\mathcal{D})$).
Le résultat s'ensuit, puisque $\check{D}_{{\rm dif},n}^+$ est engendré par $\varphi^{-n}(\check{D}^{]0,r_n]})$. 

\end{proof}
 
 Le lemme \ref{engendre} entraîne l'égalité $M_n(\mathcal{D})=\check{\mathcal{D}}_{{\rm dif},n}^+$, et la proposition \ref{mne} permet de conclure que $\mathcal{D}$ 
est de de Rham, ce qui finit la preuve de la proposition \ref{crux}.

\end{proof}

  Revenons à la preuve du théorème \ref{main1} et notons pour simplifier 
 $\Pi=\Pi_{\delta}(D)$ et $\Pi_1=\Pi_{\delta}(E)$. 
  Supposons que $E$ est de de Rham et que $\Pi_1^{\rm alg}= \Pi^{\rm alg}$. La proposition \ref{crux} montre que 
  $(\Pi_1)_c^{P-\rm alg}\subset \Pi_1^{\rm alg}=\Pi^{\rm alg}$, donc 
  $(\Pi_1)_c^{P-\rm alg}\subset \Pi_c^{P-\rm alg}$, ce qui contredit la proposition \ref{picscindee} et
  permet de conclure.
    
    Supposons maintenant que $\Pi_1^{\rm alg}\ne \Pi^{\rm alg}$ et montrons que 
    $(\Pi_1)_c^{P-\rm alg}\subset \Pi_1^{\rm alg}$ (ce qui permettra de conclure que $V(E)$ est de de Rham grâce 
    à la proposition \ref{crux}). La proposition \ref{picscindee} fournit une suite exacte de $B$-modules
    $$0\to \Pi_c^{P-\rm alg}\to (\Pi_1)_c^{P-\rm alg}\to \Pi_c^{P-\rm alg}\to 0.$$
    On note $(\Pi_1)_c^{\rm alg}$ le sous-espace de $\Pi_1^{\rm alg}$ engendré par les 
    vecteurs $(u-1)\cdot v$, avec $u\in U:=\left(\begin{smallmatrix} 1 & \qp \\0 & 1\end{smallmatrix}\right)$.
    Alors $(\Pi_1)_c^{\rm alg}\subset (\Pi_1)_c^{P-\rm alg}$ (par définition de 
    $(\Pi_1)_c^{P-\rm alg}$ et en utilisant la proposition \ref{pilc}). Nous aurons besoin du lemme suivant:
    
    \begin{lemma}\label{stupid}
      L'image de $(\Pi_1)_c^{\rm alg}$ dans $\Pi_c^{P-\rm alg}$ est non nulle et stable sous l'action 
 de $B$.
    \end{lemma}
    
   \begin{proof}
   La stabilité sous l'action de $B$ est évidente. Si l'image de $(\Pi_1)_c^{\rm alg}$
   dans $\Pi_c^{P-\rm alg}$ est nulle, alors l'image de tout vecteur $v\in \Pi_1^{\rm alg}$ dans 
   $\Pi^{\rm alg}$ est invariante par $U$. Or $(\Pi^{\rm alg})^U=0$, donc 
   on aurait $\Pi_1^{\rm alg}= \Pi^{\rm alg}$, une contradiction.
   \end{proof}
   
      Le lemme \ref{stupid} et la preuve du corollaire VI.5.9 de \cite{Cbigone} permettent de conclure que 
    $(\Pi_1)_c^{\rm alg}$ se surjecte sur $\Pi_c^{P-\rm alg}$. Ainsi, pour tout $v\in (\Pi_1)_c^{P-\rm alg}$ on peut trouver
    $v_1\in (\Pi_1)_c^{\rm alg}$ tel que $v-v_1\in \Pi_c^{P-\rm alg}$. Comme 
    $\Pi_c^{P-\rm alg}\subset \Pi^{\rm alg}$ (d'après la proposition \ref{crux}, car $D$ est de de Rham, ou en utilisant 
    le corollaire VI.5.9 de \cite{Cbigone}), on obtient $v\in \Pi_1^{\rm alg}$, d'où l'inclusion souhaitée
      $(\Pi_1)_c^{P-\rm alg}\subset \Pi_1^{\rm alg}$. Comme nous l'avons constaté, cela entraîne que $V(E)$ est de de Rham. 
      Pour conclure la preuve du théorème \ref{main1}, il reste à montrer que $\Pi_1^{\rm alg}$ est dense dans 
      $\Pi_1$. Or, nous venons de montrer qu'il contient $(\Pi_1)_c^{P-\rm alg}$. Il suffit donc de montrer que ce dernier espace
     est dense dans $\Pi_1$. En utilisant la suite exacte
     $$0\to \Pi_c^{P-\rm alg}\to \Pi_{1,c}^{P-\rm alg}\to \Pi_c^{P-\rm alg}\to 0$$
    fournie par la proposition \ref{picscindee}, on se ramène à prouver la densité de $\Pi_c^{P-\rm alg}$ dans 
    $\Pi$. Mais $\Pi_c^{P-\rm alg}$ est un sous-$B$-module non nul de $\Pi$ d'après la proposition \ref{pipalg}, le corollaire \ref{blah} 
     et la proposition \ref{tatetwist}. Or, $\Pi$ est topologiquement irréductible comme $B$-module d'après \cite[prop. III.32]{CD}. Cela finit la preuve du théorème 
     \ref{main1}.

  \subsection{Déformations de de Rham}    
  
   Rappelons que si 
    $\delta: \qpet\to L^*$ est un caractère continu (pas forcément unitaire), alors on dispose d'un 
    $(\varphi,\Gamma)$-module $\mathcal{R}(\delta)$ (qui ne correspond à une représentation galoisienne que si 
    $\delta$ est unitaire), qui  a une base $e$ telle que 
    $\varphi(e)=\delta(p)e$ et $\sigma_a(e)=\delta(a)e$ pour tout $a\in\zpet$. 
    
    \begin{definition} Soit $V$ une $L$-représentation absolument irréductible de ${\rm Gal}(\overline{\qp}/\qp)$, de dimension $2$.
    
    a) On dit que $V$ est trianguline si on peut trouver des caractères $\delta_1$ et $\delta_2$ et une suite exacte
    de $(\varphi,\Gamma)$-modules sur $\mathcal{R}$
     $$0\to \mathcal{R}(\delta_1)\to D_{\rm rig}\to \mathcal{R}(\delta_2)\to 0,$$
où $D=D(V)$. 

b) On dit que $V$ est {\it spéciale} si $V$ est trianguline, potentiellement cristalline, et si on peut trouver une suite exacte comme dans 
a) et telle que, en posant $w(V)=w(\delta_1)-w(\delta_2)$, on ait 
    $$\frac{\delta_1}{\delta_2}\in \{x^{w(V)-1}|x|, x^{w(V)-1}|x|^{-1}\}.$$

    \end{definition}
    
    Notons que $w(V)$ est la différence des poids de Hodge-Tate de 
$V$. De plus, on vérifie facilement que $V$ est spéciale si et seulement si son dual de Cartier
$V^*(1)$ est spéciale. La propriété fondamentale des représentations spéciales est le résultat suivant de Colmez:

\begin{proposition}
 Soit $\Pi\in {\rm Rep}_L(\delta)$ un objet supersingulier tel que $\Pi^{\rm alg}\ne 0$. Alors
 $\Pi^{\rm alg}$ est réductible si et seulement si $V(\Pi)$ est spéciale, dans quel cas 
 $\Pi^{\rm alg}$ est une extension de $W$ par $W\otimes {\rm St}$ pour une certaine
 représentation algébrique $W$ et où ${\rm St}$ est la représentation de Steinberg.

\end{proposition}

   On vérifie aussi que $V$ est spéciale si et seulement si $V$ est potentiellement cristalline et si la représentation de Weil-Deligne 
associée à $D_{\rm pst}(V)$ est une somme directe de deux caractères $\delta_1, \delta_2$, avec 
$\frac{\delta_1}{\delta_2}\in \{|x|, |x|^{-1}\}$.

   Si $V$ est une représentation galoisienne, on identifie librement les déformations de $V$ à $L[\varepsilon]$ et
   les extensions de $V$ par lui-même. Si $\det V=\zeta$, on dit qu'une telle déformation $V_1$ de $V$ a déterminant $\zeta$
   si $\det_{L[\varepsilon]}(V_1)=\zeta$.

\begin{definition} Soit $\zeta$ un caractère unitaire et soit $V$ une $L$-représentation de de Rham, telle que $\det V=\zeta$. 
  On note ${\rm Ext}^1_{g,\zeta}(V,V)$ le sous-espace de ${\rm Ext}^1(V,V)$ engendré par les déformations de $V$ à $L[\varepsilon]$, qui ont déterminant $\zeta$ 
   et qui sont de de Rham. 
\end{definition}

  \begin{proposition}\label{iwas} Soit $\zeta$ un caractère unitaire et soit $V$ une $L$-représentation absolument irréductible, de de Rham, à poids de Hodge-Tate distincts
  et telle que $\det V=\zeta$. 
   Alors $\dim_L {\rm Ext}^1_{ g, \zeta}(V,V)=1$
   sauf si $V$ est spéciale, auquel cas $\dim_L {\rm Ext}^1_{ g, \zeta}(V,V)=2$.
   \end{proposition}
   
   \begin{proof} On pose $W={\rm ad}^0(V)$, sous-espace de ${\rm End}(V)$ formé des endomorphismes de trace nulle.
   On a donc un isomorphisme de ${\rm Gal}(\overline{\qp}/\qp)$-représentations $V\otimes V^*=W\oplus 1$.
      Il découle de la proposition 1.26 de \cite{Nek} et de sa preuve que l'on a un isomorphisme 
   $${\rm Ext}^1_{g,\zeta}(V,V)\simeq H^1_g(\qp, W),$$
   où $H^1_g(\qp, W)$ désigne le groupe de cohomologie introduit par Bloch-Kato \cite{BK}, classifiant 
   les extensions de $\qp$ par $W$ qui sont de de Rham. Un calcul standard utilisant l'exponentielle de Bloch-Kato et l'exponentielle duale de Kato
   montre (voir par exemple la prop. 1.24 de \cite{Nek}) que 
   $$\dim H^1_g(\qp, W)=\dim H^0(\qp, W)+\dim (D_{\rm dR}(W)/D_{\rm dR}^+(W))+
   \dim D_{\rm cris}(W^*(1))^{\varphi=1}.$$
   Puisque $V$ est absolument irréductible, on a $H^0(\qp, W)=0$. 
    Ensuite, $$D_{\rm dR}^+(V\otimes V^*)=\{f\in {\rm End}(D_{\rm dR}(V))| \quad f({\rm Fil}^i(D_{\rm dR}(V)))\subset  {\rm Fil}^i(D_{\rm dR}(V)),\quad \forall i\in\mathbf{Z}\}.$$
    Comme les poids de Hodge-Tate de $V$ sont distincts et $\dim V=2$, on obtient $\dim D_{\rm dR}^+(V\otimes V^*)=3$, donc 
    $\dim (D_{\rm dR}(W)/D_{\rm dR}^+(W))=1$. La proposition est donc équivalente à l'assertion suivante: 
     $D_{\rm cris}(V\otimes V^*)^{\varphi=p}=0$ sauf si $V$ est spéciale, dans quel cas $D_{\rm cris}(V\otimes V^*)^{\varphi=p}$
     est de dimension $1$. 
     
       Supposons que $V$ n'est pas trianguline et considérons une extension finie $K$ de $\qp$ telle que 
       $V$ devienne semistable sur $K$. Puisque $V$ n'est pas trianguline, la représentation de Weil-Deligne associée à 
       $D_{\rm pst}(V)$ est irréductible, donc $V$ devient cristalline sur $K$. Alors, en posant $D^K_{\rm cris}=(B_{\rm cris}\otimes_{\qp} V)^{{\rm Gal}(\overline{\qp}/K)}$, on a 
       $$D_{\rm cris}(V\otimes V^*)^{\varphi=p}\simeq (D^{K}_{\rm cris}\otimes D^K_{\rm cris})^{{\rm Gal}(K/\qp), \varphi=p}$$
       et ce dernier espace est nul, puisque $\varphi$ a une base de vecteurs propres avec des valeurs propres de même
       valuation $p$-adique (cela découle de l'irréductibilité de 
      la représentation de Weil-Deligne
      associée à $D_{\rm pst}(V)$, ou par contemplation des formules 
      de \cite{GhateMezard}). 
       On en déduit le résultat dans ce cas.
      
        Supposons donc que $V$ est trianguline. Si $V$ est une tordue d'une représentation semi-stable non cristalline, on peut supposer
        que $V$ est semi-stable, et alors 
        $$D_{\rm cris}(V\otimes V^*)^{\varphi=p}=(D_{\rm st}(V)\otimes D_{\rm st}(V)^*)^{N=0, \varphi=p}.$$
        Dans ce cas, le $\S$ 3.1 de \cite{GhateMezard} montre que $D_{\rm st}(V)$ a une base dans laquelle
        les matrices de $\varphi$ et $N$ sont $\varphi=\left(\begin{smallmatrix} \frac{p}{\alpha} & 0 \\ 0 & \frac{1}{\alpha}\end{smallmatrix}\right)$
         et $N=\left(\begin{smallmatrix} 0 & 0 \\1 & 0\end{smallmatrix}\right)$. On vérifie alors sans mal que 
      $(D_{\rm st}(V)\otimes D_{\rm st}(V)^*)^{N=0, \varphi=p}=0$. 
      
        Enfin, supposons que $V$ est une tordue d'une représentation devenant cristalline sur une extension abélienne de $\qp$, on vérifie
        en utilisant les formules de la def. 2.4.4 de \cite{BB} et la prop. 2.4.5 de loc.cit que 
        $D_{\rm cris}(V\otimes V^*)^{\varphi=p}=0$, sauf si $V$ est spéciale, dans quel cas c'est un espace de dimension $1$.

   \end{proof}
   
   \begin{remark}
   La non nullité de $D_{\rm cris}(V\otimes V^*)^{\varphi=p}$ équivaut à l'existence d'une extension 
   $0\to V\to E\to V\to 0$ qui est de de Rham, mais pas cristalline. Si $V$ devient cristalline sur une extension $K$ (forcément abélienne
   d'après le théorème précédent), alors $E$ est semi-stable et non cristalline sur $K$.
   
   \end{remark}


\begin{thebibliography}{40}
\footnotesize

\bibitem [1] {Ber} L.~Berger-Repr\'{e}sentations $p$-adiques
 et \'{e}quations diff\'{e}rentielles, Invent. Math. 148 (2002),
 p.219-284.
 
 \bibitem [2] {BerAst} L.~Berger-Équations différentielles $p$-adiques et $(\varphi,N)$-modules filtrés,
Astérisque 319 (2008), 13--38. 

\bibitem[3]{BB} L.~Berger, C.~Breuil-Sur quelques
repr\'{e}sentations potentiellement cristallines de
$G_{\qp}$, Ast\'{e}risque 330, p. 155-211.

\bibitem[4]{BK} S.~Bloch, K.~Kato- $L$-functions and Tamagawa numbers of motives, The Grothendieck Festschrift, Vol. I, Progr. Math., vol. 86, Birkhäuser Boston. 

\bibitem[5]{BMorig}	C.~Breuil, A.~Mézard-Multiplicités modulaires et représentations de ${\rm GL}_2(\zp)$ et de ${\rm Gal}(\overline{\qp}/\qp)$ en $l=p$ (avec un appendice par G. Henniart), Duke Math. J. 115, 2002, 205-310.

\bibitem[6]{BM} C.~Breuil, A.~Mézard-Multiplicités modulaires raffinées, à paraître à Bull. Soc. Math. de France.

\bibitem [7]{Br} C.~Breuil-Invariant \textit{L}
 et s\'{e}rie sp\'{e}ciale $p$-adique, Ann. Sci. Ecole Norm. Sup
 37 (2004), p. 559-610.

\bibitem [8]{CCsurconv} F.~Cherbonnier, P.~Colmez-Repr\'{e}sentations
$p$-adiques surconvergentes, Invent. Math. 133 (1998), p. 581-611.

\bibitem[9] {Cannals} P.~Colmez- Th\'{e}orie d'Iwasawa des repr\'{e}sentations de de Rham
d'un corps local, Ann. of Math. 148 (1998), p. 485-571.

\bibitem[10]{Cmirab} P.~Colmez-$(\varphi,\Gamma)$-modules et repr\'{e}sentations
du mirabolique de $G$, Ast\'{e}risque 330 (2010), p. 61-153.

\bibitem[11]{Cbigone} P.~Colmez-Repr\'{e}sentations de ${\rm GL}_2(\qp)$ et
$(\varphi,\Gamma)$-modules, Ast\'{e}risque 330 (2010), p. 281-509.

\bibitem[12]{CvectdR} P.~Colmez-Espaces Vectoriels de dimension finie et représentations de de Rham, Astérisque 319 (2008), 117-186.


\bibitem[13]{Cvectan} {\sc P.~Colmez}, La s\'erie principale unitaire de
${\rm GL}_2(\qp)$: vecteurs localement analytiques, à paraître dans Proceedings of the Durham LMS Symposium 2011. 

\bibitem[14]{CD} P.~Colmez, G.~Dospinescu-Complétés universels de représentations de ${\rm GL}_2(\qp)$, 
preprint, disponible à \url{http://arxiv.org/pdf/1301.7647v1.pdf}. 

\bibitem[15]{Annalen} G.~Dospinescu-Actions infinit\'{e}simales dans la correspondance de Langlands
locale $p$-adique pour ${\rm GL}_2(\qp)$, Math. Annalen, \`{a} para\^{i}tre.

\bibitem[16]{Dthèse} G.~Dospinescu-Actions infinitésimales dans la correspondance de Langlands locale $p$-adique, thèse de l'Ecole Polytechnique, juin $2012$. 

\bibitem[17]{DBenjamin} {\sc G.~Dospinescu} et {\sc B.~Schraen},  Endomorphism algebras of admissible
$p$-adic representations of $p$-adic Lie groups, Representation Theory (\`a para\^\i tre).

\bibitem[18]{DIwasawa} G.~Dospinescu-Modèle de Kirillov et théorie d'Iwasawa des représentations de de Rham, en préparation.

\bibitem[19]{Dext} G.~Dospinescu-Extensions de représentations supersingulières mod $p$ et Banachiques, en préparation.

\bibitem[20]{Emcomp} M.~Emerton-Local-global compatibility in the p-adic Langlands programme for ${\rm GL}_2/\mathbf{Q}$, preprint, disponible \`{a} \url{http://www.math.northwestern.edu/~emerton/preprints.html}.

\bibitem[21]{EmCoates} M.~Emerton-A local-global compatibility conjecture in the p-adic Langlands programme for ${\rm GL}_2/\mathbf{Q}$,
Pure and Applied Math. Quarterly 2 (2006), no. 2, 279-393. 

\bibitem[22]{Emlocan} M.~Emerton-Locally analytic vectors in representations of locally p-adic analytic groups,
to appear in Memoirs of the AMS, disponible \`{a} \url{http://www.math.northwestern.edu/~emerton/preprints.html}.

\bibitem[23]{FoAnnals} J.-M.~Fontaine-Sur Certains Types de Repr\'{e}sentations p-Adiques du Groupe de Galois d'un Corps Local; Construction d'un Anneau de Barsotti-Tate, Ann. of Math. 115 (1982), p.529-577.

\bibitem[24]{FoGrot} J.-M.~Fontaine-Repr\'{e}sentations
$p$-adiques des corps locaux. I, in The
Grothendieck Festschrift, Vol II, Progr. Math.,
vol 87, Birkhauser, 1990, p. 249-309.

\bibitem[25]{GhateMezard} E.~Ghate, A.~Mézard, Filtered modules with coefficients, Transactions of the 
American Mathematical Society, vol. 361, nb. 5, May 2009, pages 2243–2261. 

\bibitem[26]{KiFM} M.~Kisin-The Fontaine-Mazur conjecture for ${\rm GL}_2$,  J.A.M.S. 22(3) (2009) 641-690. 


\bibitem[27]{LXZ} {\sc R. Liu}, {\sc B. Xie}, et {\sc Y. Zhang},
 Locally Analytic Vectors of Unitary Principal Series of $\rm{GL}_2(\qp)$, 
Ann. E.N.S.~{\bf 45} (2012), 167--190.

\bibitem[28]{Nek} J.~Nekovar- On $p$-adic height pairings, Séminaire de Théorie des Nombres, Paris, 1990-91, 127 - 202, Progr. in Math., 108, Birkhäuser Boston, Boston, MA, 1993. 


\bibitem[29]{Pa} V.~Pa\v{s}k\={u}nas-The image
of Colmez's Montr\'{e}al functor, preprint, disponible \`{a} \url{http://arxiv.org/abs/1005.2008}.


\bibitem[30]{PaskunasBM} V.~Pa\v{s}k\={u}nas-On the Breuil-Mézard conjecture, preprint, disponible à ...

\bibitem[31]{Sen} S.~Sen-Continuous cohomology and
$p$-adic Galois representations, Invent. Math 62 (1980/1981),
p. 89-116.


\bibitem[32]{STInv} P.~Schneider, J.~Teitelbaum-
Algebras of $p$-adic distributions and admissible
representations, Invent. Math. 153 (2003), p. 145-196.



\end{thebibliography}
\end{document}